\documentclass[12pt,twoside,english,a4paper]{article}
\usepackage{babel,amssymb,amsthm}
\usepackage{graphicx}
\usepackage{amsmath}
\usepackage{bm}
\usepackage[show]{ed}

\newcommand{\punto}{\,\cdot\,}
\newcommand{\ds}{\displaystyle}
\newcommand{\smallfrac}[2]{{\textstyle\frac{#1}{#2}}} 
\newcommand{\jump}[1]{[\![#1]\!]}
\newcommand{\ave}[1]{\{\!\!\{#1\}\!\!\}}

\newtheorem{proposition}{Proposition}[section]
\newtheorem{corollary}[proposition]{Corollary}
\newtheorem{lemma}[proposition]{Lemma}

\numberwithin{equation}{section}

\title{Variational views of stokeslets and stresslets}

\date{\today}

\author{Francisco--Javier Sayas\footnote{E-mail: {\tt fjsayas@udel.edu} --  Partially supported by the NSF (DMS 1216356).}\\
Department of Mathematical Sciences \\
University of Delaware, Newark DE, USA\\
\& \\
Virginia Selgas\footnote{E-mail: {\tt selgasvirginia@uniovi.es} -- Partially supported by MICINN (Project MTM2010-21135-C021-01) and the Universidad de Oviedo `Ayudas de Movilidad de Excelencia' Program.}\\
Departamento de Matem\'aticas\\
Universidad de Oviedo, Spain}

\begin{document}

\maketitle

\begin{abstract}
In this paper we present a self-contained variational theory of the layer potentials for the Stokes problem on Lipschitz boundaries. We use these weak definitions to show how to prove the main theorems about the associated Calder\'on projector. Finally, we relate these variational definitions to the integral forms. Instead of working these relations from scratch, we show some  formulas parametrizing the Stokes layer potentials in terms of those for the Lam\'e and Laplace operators. While all the results in this paper are well known for smooth domains, and most might be known for non-smooth domains, the approach is novel a gives a solid structure to the theory of Stokes layer potentials.
\end{abstract}

\section{Motivation and basic notation}

In this paper we present a coherent and self-contained theory of the layer potentials for the Stokes equation on Lipschitz boundaries, in two and three dimensions. The core of the paper is the presentation of the results in variational form, using a formalization based on weighted Sobolev spaces. This presentation, {\em \` a la N\'ed\'elec}, follows the way how the properties of integral operators/equations (in particular coercivity of equations of the first kind) were introduced in the often quoted and not easy to find lecture notes of Jean-Claude N\'ed\'elec \cite{Nedelec:1971}. What is missing there is a rigorous proof that the potentials correspond to the integral expressions that have been known of old (with some work, this recognition can be done for smooth enough domains). A way to solve this problem is to start from the other side, defining the potentials in integral form, and then showing that they satisfy the variational equations that were used for the weak definition of potentials. The definition of potentials, in weak form but using integral expressions, is the modern standard for their introduction, or better said, it is the standard that numerical analysts have accepted as the one that satisfies their theoretical needs. A big part of the boundary integral community has embraced William McLean's monograph \cite{McLean:2000} as the answer to their prayers, the book that proves everything, the source of all needed references. However general it is, the general frame that is developed in \cite{McLean:2000} does not include the Stokes equation. We are actually going to openly borrow results from this text, in a slightly surprising way, at the very end of this paper. The long awaited --and finally appeared in the inmensity of its six hundred page format--  book of George Hsiao and Wolfgang Wendland \cite{HsWe:2008} deals with many more equations, but in its desire to enjoy the full power of the pseudodifferential calculus, it does not include many results (and/or tools) to deal with layer potentials on, say, polyhedra. This fast review would be definitely incomplete if it were not to mention a third way, which has proven to be the one providing the best mapping properties for boundary integral operators. The techniques employed in this way are much more sophisticated than what we are going to see in this paper, picking from advanced analysis of singular integrals. The Chicago-Minnesota school of harmonic analysis (Eugene Fabes, Carlos Kenig, Gregory Verchota \cite{FaKeVe:1988}) started the development of this kind of treatment of the Stokes integral operators.

This is what we are going to do next. We are going to use the variational technique of N\'ed\'elec to introduce the layer potentials  on Lipschitz surfaces for the Stokes problem in two and three dimensions and we are going to use these definitions to prove the main properties of the associated boundary integral operators. Next, we are going to reconstruct the potentials using integral formulas. With combinations of Lam\'e and Laplace potentials, plus the standard properties thereof \cite{McLean:2000}, we will show that integral potentials and variational potentials are one and the same. 

Some (or even all) of the results of this paper are known, but not well known, and definitely not easy to find in the literature. The presentation is original (up to the fact that we mimick the French school variational method) and we make a point in solving some difficulties in the two dimensional case that otherwise lead to a definition of potentials in a quotient space. To the best of our knowledge, we are novel in integrating the two approaches in one unique proof format. In any case, the goal of the paper is twofold: give a clean presentation of techniques and easy-to-refer results; help the reader get acquainted with the mathematics behind the theory of layer potentials.

Apart from well-established background in basic functional analysis, Sobolev space theory and elliptic PDE, we will need some popular results (Korn's inequalities and the theory of abstract mixed problems, a.k.a., the Brezzi conditions) and a few more ingredients:
\begin{itemize}
\item Of the big collections of weighted Sobolev spaces developed by Bernard Hanouzet and Jean Giroire, we will employ one of each, with the aim of controlling behavior at infinity of potentials with local $H^1$ regularity. The results that are needed are taken from \cite{AmGiGi:1994, Hanouzet:1971}.
\item A result concerning the surjectivity of the divergence operator from the previous weighted Sobolev spaces to $L^2$ will be used. This was proved by Vivette Girault and Ad\'elia Sequeira \cite{GiSe:1991}.
\item Finally, the integral theory of layer potentials for the Laplace and Lam\'e operators will be used at some key points at the end of this paper. The results will be taken from William  McLean's monograph \cite{McLean:2000}, which exploits the ideas of a celebrated paper by Martin Costabel \cite{Costabel:1988}. (Note again that the Stokes operator is not a particular case of the very general theory developed in \cite{McLean:2000}.)
\end{itemize}
Precise reference on all the results will be given as we reach them. We next list the basic background material and the general notation that will be used throughout. All elements of this list are standard and the reader can use this collection of bullet points for easy reference.
\begin{description}
\item {\em Functional analysis.} If $H$ is a Hilbert space, $\mathbf H:=H^d$ endowed with the product norm: $d\in \{2,3\}$ will be space dimension. If $H$ is a Hilbert space and $X\subset H$, then $H'$ is the dual space of $H$ and $X^\circ =\{\ell \in H'\,:\,\ell (x)=0 \quad \forall x \in X\}$ is the polar set of $X$.
\item {\em Scalars, vectors, and matrices.}  We will  keep separate notation for scalar (italic, as in $u$), vector (boldface, as in $\mathbf u$) and matrices-tensors (capitals, as in $\mathrm U$) quantities. In case of doubt, {\em vectors are column vectors}.
The colon will be used for the tensor (Frobenius inner) product of matrices $\mathrm U:\mathrm V=\sum_{ij} \mathrm U_{ij}\mathrm V_{ij}$. Everything will be real-valued and there will be no need for conjugation. The tensor product of two non-zero vectors $\mathbf u$ and $\mathbf v$ is the rank one matrix $\mathbf u\otimes\mathbf v$ with entries $u_i\,v_j$ and the following formula should be kept in mind: $(\mathbf u\otimes\mathbf v)\mathbf w=(\mathbf v\cdot\mathbf w)\,\mathbf u$.
\item {\em $L^2$ products.} Parentheses will be used for $L^2$ products:
\[
(u,v)_{\mathcal O}:=\int_{\mathcal O} u\,v,\qquad (\mathbf u,\mathbf v)_{\mathcal O}=\int_{\mathcal O}\mathbf u\cdot\mathbf v, \qquad (\mathrm U,\mathrm V)_{\mathcal O}:=\int_{\mathcal O} \mathrm U :\mathrm V.
\]
The associated norms will be denoted $\|\punto\|_{\mathcal O}$.
\item {\em Multivariable calculus.} Given a vector field $\mathbf u$, we define $\mathrm D\mathbf u$ and $\boldsymbol\varepsilon(\mathbf u)$ by
\[
(\mathrm D \mathbf u)_{ij}:=\frac{\partial u_i}{\partial x_j}, \qquad \boldsymbol\varepsilon(\mathbf u):=\smallfrac12 (\mathrm D \mathbf u+(\mathrm D\mathbf u)^\top).
\]
Note the useful formula $\boldsymbol\varepsilon(\mathbf u):\boldsymbol\varepsilon(\mathbf v)=\boldsymbol\varepsilon(\mathbf u):\mathrm D \mathbf v$. If $\mathrm A$ is matrix valued, then $\mathrm{div}\,\mathrm A$ is the divergence operator applied to the rows of $\mathrm A$ (so that the output is a vector). This is consistent with the formula
\[
\mathrm{div}(\mathrm A^\top \mathbf v)=(\mathrm{div}\,\mathrm A) \cdot\mathbf v+\mathrm A : \mathrm D \mathbf v,
\]
which is in the heart of the integration by parts process.
\item {\em The geometric setting.} In what follows $\Omega_-$ is a connected bounded Lipschitz open set  boundary $\Gamma$. The exterior domain is $\Omega_+:=\mathbb R^d\setminus\overline{\Omega_-}$. The case of boundaries with multiple connected components is very easy to handle, but we will prefer to stay in the connected case.
\item {\em Sobolev theory.} For a bounded or unbounded domain on one side of its Lipschitz boundary, $H^1(\mathcal O)$ is the usual Sobolev space,
$\gamma:H^1(\mathcal O)\to H^{1/2}(\partial O)$ is the trace operator, and $H^1_0(\mathcal O)=\mathrm{Ker}\,\gamma$.
When we identify $L^2(\partial\mathcal O)$ with its dual space, the dual space of $H^{1/2}(\partial\mathcal O)$ is automatically represented by a superset of $L^2(\partial\mathcal O)$. This representation is called $H^{-1/2}(\partial\mathcal O)$. The duality product $H^{-1/2}(\partial\mathcal O)\times H^{1/2}(\partial\mathcal O)$ is denoted $\langle\cdot,\cdot\rangle_{\partial\mathcal O}$. For matrix valued $\mathrm A\in H^{1/2}(\partial\mathcal O)^{d\times d}$ and $\boldsymbol\xi\in \mathbf H^{-1/2}(\partial\mathcal O)$, we will use
$\langle \mathrm A,\boldsymbol\xi\rangle_{\partial\mathcal O}$ to denote  the vector with components $\sum_{j=1}^d \langle \xi_j, \mathrm A_{ij}\rangle_{\partial\mathcal O}.$
The space of $\mathcal C^\infty$ functions with compact support contained in $\mathcal O$ will be denoted $\mathcal D(\mathcal O)$.
\item {\em Local behavior.} The subscript $\mathrm{loc}$ will be used in some parts of this document. For instance, when we write $u \in H^1_{\mathrm{loc}}(\mathcal O)$, we mean that $\varphi u\in H^1(\mathcal O)$ for all $\varphi \in \mathcal D(\mathbb R^d)$ (not only $\varphi\in \mathcal D(\mathcal O)$). This means that we are only localizing at infinity and we are not separating ourselves from the boundary of the set.
\end{description}

\section{The Sobolev spaces of Giroire and Hanouzet}

For an open set $\mathcal O\subset \mathbb R^d$, we consider the
 weighted Sobolev space
\[
W(\mathcal O):=\{ u :\mathcal O\to\mathbb R\,:\, \rho\, u\in L^2(\mathcal O), \nabla u\in \mathbf L^2(\mathcal O)\},
\]
endowed with its natural norm
\[
\|u\|_{1,\rho,\mathcal O}^2:=\| \rho \, u\|_{\mathcal O}^2+\|\nabla u\|_{\mathcal O}^2,
\]
where $\rho=\rho_d$ is given by:
\[
\rho_3(\mathbf x ):= \frac1{\sqrt{1+|\mathbf x|^2}}, \qquad \rho_2(\mathbf x):=\frac1{1+\smallfrac12 \log(1+|\mathbf x|^2)}\,\frac1{\sqrt{1+|\mathbf x|^2}}.
\]
Note that if $\mathcal O$ is bounded $H^1(\mathcal O)=W(\mathcal O)$. It is an easy exercise to prove that if $\mathcal O$ is unbounded and $ P_k(\mathcal O)$ is the space of $d$-variate polynomials of degree no larger than $k$, then
\begin{equation}\label{eq:2.A1}
W(\mathcal O)\cap   P_k(\mathcal O)=\left\{ \begin{array}{ll} \{ 0\}, & \mbox{if $d=3$},\\ \  P_0(\mathcal O), &\mbox{if $d=2$},\end{array}\right.
\end{equation}
that is, $W(\mathcal O)$ does not contain non-trivial polynomials in three dimensions, but it contains constant functions in two dimensions.

\begin{proposition}\label{prop:2.1}
The space $\mathcal D(\mathbb R^d)$ is dense in $W(\mathbb R^d)$. Moreover, the following norms are equivalent to the norm $\|\punto\|_{1,\rho,\mathbb R^d}$ in $W(\mathbb R^d)$:
\begin{itemize}
\item[{\rm (a)}] $\|\nabla\cdot\|_{\mathbb R^3}$ (for $d=3$).
\item[{\rm (b)}] $\big(\|\nabla\cdot\|_{\mathbb R^2}^2+|\jmath(\cdot)|^2\big)^{1/2}$ (for $d=2$).
\end{itemize}
In (b),  $\jmath:W(\mathbb R^2)\to \mathbb R$ is a bounded functional such that $\jmath(1)\neq 0$.
\end{proposition}

\begin{proof} The density result in the case $d=3$ is easy \cite[Th\'eor\` eme I.1]{Hanouzet:1971}. The case $d=2$ can be found in \cite{AmGiGi:1992}: it requires to first show that $H^1(\mathbb R^2)$ is dense in $W(\mathbb R^2)$ using a clever damping argument. Part (a) is proved in \cite[Th\'eor\`eme I.2]{Hanouzet:1971}. Part (b) can be found in \cite{AmGiGi:1992}, stated in equivalent form in a quotient space: $\|\nabla\cdot\|_{\mathbb R^2}$ is an equivalent norm in $W(\mathbb R^2)/P_0$. Note that the references for the two dimensional results are more recent than the results themselves, which were already mentioned in \cite{LeRoux:1974}.\end{proof}

\begin{corollary}\label{cor:2.2}
The following norms are equivalent to  the product norm in $\mathbf W(\mathbb R^d)$:
\begin{itemize}
\item[{\rm (a)}] $\|\boldsymbol\varepsilon(\punto)\|_{\mathbb R^3}$ (for $d=3$).
\item[{\rm (b)}] $\big(\|\boldsymbol\varepsilon(\punto)\|_{\mathbb R^2}^2+|\boldsymbol\jmath(\cdot)|^2\big)^{1/2}$ (for $d=2$).
\end{itemize}
In {\rm (b)} $\boldsymbol\jmath(\mathbf u):=(\jmath(u_1),\jmath(u_2))$ and $\jmath:W(\mathbb R^2)\to \mathbb R$ is a bounded functional such that $\jmath(1)\neq 0$.
\end{corollary}

\begin{proof} For all $\mathbf u \in\boldsymbol{\mathcal D} (\mathbb R^d)$ we can write
\[
2(\boldsymbol\varepsilon(\mathbf u),\,\boldsymbol\varepsilon(\mathbf u)\,)_{\mathbb R^d}= 2(\boldsymbol\varepsilon(\mathbf u),\,\mathrm D\mathbf u\,)_{\mathbb R^d}=(\mathrm D \mathbf u,\,\mathrm D\mathbf u)_{\mathbb R^d}+(\mathrm{div}\,\mathbf u,\mathrm{div}\,\mathbf u)_{\mathbb R^d}.
\]
Therefore
\[
\|\mathrm D \mathbf u\|^2_{\mathbb R^d} \le 2\|\boldsymbol\varepsilon(\mathbf u)\|_{\mathbb R^d}^2=\smallfrac12 \|\mathrm D \mathbf u+(\mathrm D\mathbf u)^\top\|_{\mathbb R^d}^2\le  2 \|\mathrm D\mathbf u\|_{\mathbb R^d}^2\qquad \forall \mathbf u \in\boldsymbol{\mathcal D} (\mathbb R^d).
\]
By density (Proposition \ref{prop:2.1}) it follows that
\begin{equation}\label{eq:2.1}
\smallfrac1{\sqrt2}\|\mathrm D \mathbf u\|_{\mathbb R^d}\le \|\boldsymbol\varepsilon(\mathbf u)\|_{\mathbb R^d}\le \|\mathrm D \mathbf u\|_{\mathbb R^d}\qquad \forall\mathbf u\in \mathbf W(\mathbb R^d).
\end{equation}
The result is then a direct consequence of \eqref{eq:2.1} and Proposition \ref{prop:2.1}.
\end{proof}

Because the spaces $W(\mathcal O)$ are locally $H^1(\mathcal O)$, we can define the interior and exterior traces
\[
\gamma^+:W(\Omega_+) \to H^{1/2}(\Gamma), \qquad \gamma^-:H^1(\Omega_-)\to H^{1/2}(\Gamma).
\]
These trace operators are surjective. 

\begin{proposition}\label{prop:2.3}
The space $\mathcal D(\Omega_+)$ is dense in $W_0(\Omega_+):=\{ u\in W(\Omega_+)\,:\, \gamma_+ u =0\}$.
\end{proposition}

\begin{proof} This result is a straightforward consequence of Proposition \ref{prop:2.1} and the fact that $\mathcal D(\mathcal O)$ is dense in the kernel of the trace operator for any Lipschitz domain $\mathcal O$. \end{proof}

\begin{proposition}\label{cor:2.5}
The divergence operator $\mathrm{div}:\mathbf W(\mathbb R^d)\to L^2(\mathbb R^d)$ is surjective.
\end{proposition}

\begin{proof}
The divergence operator $\mathrm{div}:\mathbf W_0(\Omega_+) \to L^2(\Omega_+)$ is surjective. This is \cite[Theorem 3.3]{GiSe:1991}.
Given $p\in L^2(\mathbb R^d)$ we take $p_-:=p|_{\Omega_-}$ and choose $\mathbf u_-\in \mathbf H^1(\Omega_-)$ such that $\mathrm{div}\,\mathbf u_-=p_-$ in $\Omega_-$. Let $\widetilde{\mathbf u_-}\in \mathbf H^1(\mathbb R^d)$ be such that $\widetilde{\mathbf u_-}|_{\Omega_-}=\mathbf u_-$. We then  build 
$\mathbf v_+\in \mathbf W_0(\Omega_+) $ such that $\mathrm{div}\,\mathbf v_+=p|_{\Omega_+}-\mathrm{div} \widetilde{\mathbf u_-}|_{\Omega_+}$. Finally, we consider
\[
\mathbf u:=\left\{\begin{array}{ll} \mathbf u_-, & \mbox{in $\Omega_-$},\\ \widetilde{\mathbf u_-}|_{\Omega_+}+\mathbf v_+, & \mbox{in $\Omega_+$}.\end{array}\right.
\]
Then $\mathbf u\in \mathbf W(\mathbb R^d)$ and $\mathrm{div}\,\mathbf u=p$.
\end{proof}

Consider the spaces $\mathcal M:=\mathcal M_d$:
\begin{eqnarray*}
\mathcal M_3 &:=&\{\mathbf a+\mathbf b\times\mathbf x\,:\, \mathbf a,\mathbf b\in \mathbb R^3\},\\
\mathcal M_2&:=& \{\mathbf a+b\,(x_2,-x_1)\,:\, \mathbf a\in \mathbb R^2, b\in \mathbb R\}.
\end{eqnarray*}
Note that if $\mathbf v \in \mathcal M$, then $\boldsymbol\varepsilon(\mathbf v)=0$ and $\mathrm{div}\,\mathbf v=0$.
Also $\mathrm{dim}\,\mathcal M_3=6$, $\mathrm{dim}\,\mathcal M_2=3$, and $\mathcal M\cap \mathbf H^1_0(\Omega_-)=\{\mathbf 0\}.$
Therefore, the space  $\mathcal M_\Gamma:=\gamma\mathcal M$ has the same dimension as $\mathcal M$ and elements of $\mathcal M$ are determined by their values on $\Gamma$, which will allow us to just identify $\mathcal M_\Gamma\equiv\mathcal M$. 
Finally, by \eqref{eq:2.A1}
\begin{equation}\label{eq:2.30}
\mathcal M \cap \mathbf W(\Omega_+)=\left\{ \begin{array}{ll} \{ \mathbf 0\}, & \mbox{if $d=3$},\\ \boldsymbol P_0(\Omega_+), &\mbox{if $d=2$}.\end{array}\right.
\end{equation}

\begin{lemma}[Rigid motions and Korn's inequality]\label{lemma:rigid}
Let $\mathcal O$ be a bounded domain.
\begin{itemize}
\item[{\rm (a)}] If $\mathbf v\in \mathbf H^1(\mathcal O)$ satisfies $\boldsymbol\varepsilon(\mathbf v)=0$ and $\mathcal O$ is connected, then $\mathbf v\in \mathcal M$.
\item[{\rm (b)}] There exist constants $C,c>0$ such that
\[
\|\boldsymbol\varepsilon(\mathbf u)\|_{\mathcal O}^2 + c\|\mathbf u\|_{\mathcal O}^2\ge C \|\mathrm D\mathbf u\|_{\mathcal O}^2 \qquad \forall \mathbf u\in \mathbf H^1(\mathcal O).
\]
\item[{\rm (c)}] If $m:=\mathrm{dim}\,\mathcal M_d$ and $\boldsymbol\jmath:\mathbf H^1(\mathcal O)\to \mathbb R^m$ is linear, bounded, and such that $\boldsymbol\jmath(\mathbf c)\neq \mathbf 0$ for all $\mathbf 0\neq \mathbf c\in \mathcal M_d$, then
\[
\Big(\|\boldsymbol\varepsilon(\mathbf u)\|_{\mathcal O}^2 +|\boldsymbol\jmath(\mathbf u)|^2\Big)^{1/2}
\]
is an equivalent norm in $\mathbf H^1(\mathcal O)$.
\end{itemize}
\end{lemma}

\begin{proof}
The proof of Part (a) for $d=3$ is given as \cite[Lemma 10.5]{McLean:2000}. The two dimensional case is a simple consequence of the three dimensional case, by immersing the two dimensional vector field in a space of three dimensional vector fields. Part (b) is the well known Korn's second inequality \cite[Theorem 10.2]{McLean:2000}.
Part (c) follows from (b) and a compactness argument.
\end{proof}

\begin{lemma}\label{prop:8.4b}
If $\mathbf u\in \mathbf W(\Omega_+)$ satisfies $\boldsymbol\varepsilon(\mathbf u)=0$, then $\mathbf u\in \mathcal M \cap \mathbf W(\Omega_+)$.
\end{lemma}

\begin{proof} Consider a domain $D_0:=B(\mathbf 0;R_0)\setminus\overline{\Omega_-}=B(\mathbf 0;R_0)\cap \Omega_+$ for $R_0$ large enough, and consider the annular domains $D_n:=\{ \mathbf x\,:\, R_0+n<|\mathbf x|<R_0+n+1\}$. Then, by Lemma \ref{lemma:rigid}(a), there exists $\mathbf c_n\in \mathcal M$ such that $\mathbf u|_{D_n}\equiv \mathbf c_n$. Comparing the traces of 
$\mathbf c_{n-1}$ and $\mathbf c_n$ on $\partial B(\mathbf 0;R_0+n)$, we show that $\mathbf c_n\equiv\mathbf c_{n+1}$. Therefore $\mathbf u\equiv \mathbf c\in \mathcal M$ in $\Omega_+$.
\end{proof}

\begin{proposition}\label{prop:2.8}
The following norms are equivalent to  the product norm in $\mathbf W(\Omega_+)$.
\begin{itemize}
\item[{\rm (a)}] $\|\boldsymbol\varepsilon(\punto)\|_{\Omega_+}$ (for $d=3$).
\item[{\rm (b)}] $\big(\|\boldsymbol\varepsilon(\punto)\|_{\Omega_+}^2+|\boldsymbol\jmath(\cdot)|^2\big)^{1/2}$ (for $d=2$).
\end{itemize}
In {\rm (b)} $\boldsymbol\jmath(\mathbf u):=(\jmath(u_1),\jmath(u_2))$ and $\jmath:W(\mathbb R^2)\to \mathbb R$ is a bounded functional such that $\jmath(1)\neq 0$.
\end{proposition}

\begin{proof}
Take $\varphi\in \mathcal D(\mathbb R^3)$ such that $\varphi\equiv 1$ in a neighborhood of $\overline{\Omega_-}$ and $\varphi$ is supported inside a ball $B$. Decomposing $\mathbf u=(1-\varphi)\mathbf u+\varphi \mathbf u$, we can understand $(1-\varphi)\mathbf u$ as an element of $\mathbf W(\mathbb R^d)$, while $\varphi\mathbf u$ is compactly supported in a ball $B$. We let $B_+:=B\cap \Omega_+$. 

We can apply Corollary \ref{cor:2.2} using $\jmath(u):=\int_\Gamma \gamma u$ for the moment being.
Simple computations show then that
\begin{eqnarray*}
\|\mathbf u\|_{1,\rho,\Omega_+} &\le & \|(1-\varphi)\mathbf u\|_{1,\rho,\mathbb R^d}+ \|\varphi \mathbf u\|_{1,\rho,B_+}\le  C_1 \| \boldsymbol\varepsilon((1-\varphi)\mathbf u)\|_{\mathbb R^d}+C_2\|\varphi\mathbf u\|_{1,B_+}\\
& \le & C_1\|\boldsymbol\varepsilon(\mathbf u)\|_{\Omega_+}+C_3\|\mathbf u\|_{B_+}+C_4\|\mathbf u\|_{1,B_+}.
\end{eqnarray*}
Here we have used that $\jmath((1-\varphi) u_i)=0$ and
\[
\boldsymbol\varepsilon((1-\varphi)\mathbf u)=(1-\varphi)\boldsymbol\varepsilon(\mathbf u)-\smallfrac12\mathbf u\otimes \nabla\varphi-\smallfrac12\nabla\varphi\otimes\mathbf u
\]
Applying Korn's inequality (Lemma \ref{lemma:rigid}(b)) in $B_+$, we can bound
\[
\|\mathbf u\|_{1,B_+}\le C_5\Big(\|\boldsymbol\varepsilon(\mathbf u)\|_{B_+}+\|\mathbf u\|_{B_+}\Big)
\]
and therefore
\[
\|\mathbf u\|_{1,\rho,\Omega_+}\le C_6 \Big( \|\boldsymbol\varepsilon(\mathbf u)\|_{\Omega_+}^2+\|\mathbf u\|_{B_+}^2\Big)^{1/2}.
\]
The remaining part of the argument is slightly different in two and three dimensions:
\begin{itemize}
\item[($d=3$)]
By Proposition \ref{prop:8.4b} and \eqref{eq:2.30}, $\|\boldsymbol\varepsilon(\punto)\|_{\Omega_+}$ is a norm in $\mathbf W(\Omega_+)$ which admits the usual norm as an upper bound. On the other hand the compact term $\|\punto\|_{B_+}$ (note that $W(B_+)=H^1(B_+)$ is compactly embedded in $L^2(B_+)$) makes this norm equivalent to the usual one. This means that we can take this compact term out of the norm and we still have an equivalent norm.
\item[($d=2$)]
By Proposition \ref{prop:8.4b} and \eqref{eq:2.30}, $\big(\|\boldsymbol\varepsilon(\punto)\|_{\Omega_+}^2+|\boldsymbol\jmath(\cdot)|^2\big)^{1/2}$ is a norm in $\mathbf W(\Omega_+)$ and we can bound
\[
\|\mathbf u\|_{1,\rho,\Omega_+}\le C_6 \Big( \|\boldsymbol\varepsilon(\mathbf u)\|_{\Omega_+}^2+|\boldsymbol\jmath(\mathbf u)|^2+\|\mathbf u\|_{B_+}^2\Big)^{1/2}\le C_7 \|\mathbf u\|_{1,\rho,\Omega_+}.
\]
Using the same argument as in the three dimensional case, we can eliminate the $\|\punto\|_{B_+}$ term and still get an equivalent norm.
\end{itemize}
\end{proof}

\section{Review \# 1: normal stresses}

Consider a matrix valued function $\boldsymbol\sigma$ (stress) such that
\[
\boldsymbol\sigma \in L^2_{\mathrm{loc}}(\mathbb R^d)^{d\times d}, \qquad \mathrm{div}\boldsymbol\sigma \in \mathbf L^2_{\mathrm{loc}}(\mathbb R^d\setminus\Gamma).
\]
We can now choose an arbitrary ball $B$ containing $\overline\Omega_-$ and write $B_+:=\Omega_+\cap B$. Then we define:
\begin{eqnarray*}
\langle \boldsymbol\sigma^-\mathbf n,\gamma \mathbf v \rangle_\Gamma &:=&(\boldsymbol\sigma,\mathrm D \mathbf v)_{\Omega_-}-(\mathrm{div}\,\boldsymbol\sigma,\mathbf v)_{\Omega_-} \qquad \forall \mathbf v\in \mathbf H^1(\Omega_-),\\
\langle \boldsymbol\sigma^+\mathbf n,\gamma \mathbf v \rangle_\Gamma &:=&-(\boldsymbol\sigma,\mathrm D \mathbf v)_{B_+}+(\mathrm{div}\,\boldsymbol\sigma,\mathbf v)_{B_+} \qquad \forall \mathbf v\in \mathbf H^1_0(B).
\end{eqnarray*}
These are matrix extensions of the classical definitions of normal traces in $\mathbf H(\mathrm{div},\Omega_-)$ and $\mathbf H_{\mathrm{loc}}(\mathrm{div},\Omega_+)$. With pretty much the same proof as in the scalar case, you show that $\boldsymbol\sigma^\pm\mathbf n\in \mathbf H^{-1/2}(\Gamma)$ and that we can bound
\[
\|\boldsymbol\sigma^-\mathbf n\|_{-1/2,\Gamma} \le C(\|\boldsymbol\sigma\|_{\Omega_-}+\|\mathrm{div}\,\boldsymbol\sigma\|_{\Omega_-}),\qquad 
\|\boldsymbol\sigma^-\mathbf n\|_{-1/2,\Gamma} \le C(\|\boldsymbol\sigma\|_{B_+}+\|\mathrm{div}\,\boldsymbol\sigma\|_{B_+}).
\]
We can also write $\jump{\boldsymbol\sigma\mathbf n}:=\boldsymbol\sigma^-\mathbf n-\boldsymbol\sigma^+\mathbf n$, so that
\[
\langle \jump{\sigma\mathbf n},\gamma \mathbf v \rangle_\Gamma =(\boldsymbol\sigma,\mathrm D \mathbf v)_{\mathbb R^d\setminus\Gamma}-(\mathrm{div}\,\boldsymbol\sigma,\mathbf v)_{\mathbb R^d\setminus\Gamma} \qquad \forall \mathbf v\in \mathbf H^1_0(B).
\]
 If $\boldsymbol\sigma \in L^2(\mathbb R^d)^{d\times d},$ and $ \mathrm{div}\boldsymbol\sigma \in \mathbf L^2(\mathbb R^d\setminus\Gamma)$, then we can apply a density argument and show that
\[
\langle \jump{\sigma\mathbf n},\gamma \mathbf v \rangle_\Gamma =(\boldsymbol\sigma,\mathrm D \mathbf v)_{\mathbb R^d\setminus\Gamma}-(\mathrm{div}\,\boldsymbol\sigma,\mathbf v)_{\mathbb R^d\setminus\Gamma} \qquad \forall \mathbf v\in \mathbf H^1(\mathbb R^d).
\]

\paragraph{Example \# 1: Laplace.}
If $\mathbf u\in \mathbf H^1_{\mathrm{loc}}(\mathbb R^d\setminus\Gamma)$ and $\Delta \mathbf u=0$ in $\mathbb R^d\setminus\Gamma$, choosing
\begin{equation}\label{stress:2}
\boldsymbol\sigma^\Delta (\mathbf u):=\mathrm D \mathbf u,
\end{equation}
we get to the integration by parts formula applied to the components of $\mathbf u$
\begin{equation}\label{stress:3}
\langle \jump{\sigma^\Delta\mathbf n},\gamma \mathbf v \rangle_\Gamma =\langle\jump{\partial_n \mathbf u},\gamma\mathbf v\rangle_\Gamma=(\mathrm D\mathbf u,\mathrm D \mathbf v)_{\mathbb R^d} \qquad \forall \mathbf v\in \mathbf H^1_0(B).
\end{equation}
This same definition can be applied for the non-homogeneous case, when $\Delta_\pm \mathbf u \in \mathbf L^2_{\mathrm{loc}}(\mathbb R^d\setminus\Gamma)$ (by $\Delta_\pm$ we are meaning the Laplace operator applied in $\Omega_\pm$, not in $\mathbb R^d$).

\paragraph{Example \# 2: Lam\'e.}  If $\mathbf u\in \mathbf H^1_{\mathrm{loc}}(\mathbb R^d\setminus\Gamma)$ , we define
\begin{equation}\label{stress:4}
\boldsymbol\sigma^{\mathrm L}(\mathbf u):=2\mu\boldsymbol\varepsilon(\mathbf u)+\lambda \,(\mathrm{div}\,\mathbf u)\,\mathrm I
\end{equation}
and we assume that $\mathrm{div}\,\boldsymbol\sigma^{\mathrm L}(\mathbf u)=\mathbf 0$ in $\mathbb R^d\setminus\Gamma$ (that is, $\mathbf u$ is a solution of the Lam\'e equations in $\mathbb R^d\setminus\Gamma$), then we arrive to Betti's formula
\begin{eqnarray}\label{stress:5}
\langle\jump{\boldsymbol\sigma^{\mathrm L}(\mathbf u)\mathbf n},\gamma\mathbf v\rangle_\Gamma &=& (\boldsymbol\sigma^{\mathrm L}(\mathbf u),\boldsymbol\varepsilon(\mathbf v))_{\mathbb R^d\setminus\Gamma}\\
&=&
2\mu (\boldsymbol\varepsilon(\mathbf u),\boldsymbol\varepsilon(\mathbf v))_{\mathbb R^d\setminus\Gamma}+\lambda (\mathrm{div}\,\mathbf u,\mathrm{div}\,\mathbf v)_{\mathbb R^d\setminus\Gamma}\qquad \forall \mathbf v\in \mathbf H^1_0(B).\nonumber
\end{eqnarray}

\paragraph{Example \# 3: Stokes.} If $(\mathbf u,p)\in \mathbf H^1_{\mathrm{loc}}(\mathbb R^d\setminus\Gamma)\times L^2_{\mathrm{loc}}(\mathbb R^d)$, we define
\begin{equation}\label{stress:6}
\boldsymbol\sigma^{\mathrm S}(\mathbf u,p):=2\nu\boldsymbol\varepsilon(\mathbf u)-p\,\mathrm I,
\end{equation}
and assume that $\mathrm{div}\,\boldsymbol\sigma^{\mathrm{S}}(\mathbf u,p)=\mathbf 0$ in $\mathbb R^d\setminus\Gamma$, we obtain a definition of jump of the normal stress for the Stokes problem.
Since this is a paper on the exterior Stokes problem, let us detail what we obtain in this case. The interior and exterior normal stresses for homogeneous solutions of the Stokes problem 
\[
(\mathbf u,p)\in \mathbf W(\mathbb R^d\setminus\Gamma)\times L^2(\mathbb R^d)\qquad \mbox{such that} \qquad -2\nu\,\mathrm{div}\,\boldsymbol\varepsilon(\mathbf u)+\nabla p =\mathbf  0 \qquad  \mbox{in $\mathbb R^d\setminus\Gamma$},
\]
are then given by $\mathbf t^\pm(\mathbf u,p)=\boldsymbol\sigma^{\mathrm S}(\mathbf u,p)\mathbf n$, i.e.,
\begin{eqnarray*}
\langle \mathbf t^-(\mathbf u,p),\gamma\mathbf v\rangle_\Gamma &=& 2\nu \left( \boldsymbol\varepsilon(\mathbf u),\boldsymbol\varepsilon(\mathbf v)\right)_{\Omega_-}-(p,\mathrm{div}\,\mathbf v)_{\Omega_-}\qquad \forall \mathbf v\in \mathbf H^1(\Omega_-),\\
\langle \mathbf t^+(\mathbf u,p),\gamma\mathbf v\rangle_\Gamma &=& -2\nu \left( \boldsymbol\varepsilon(\mathbf u),\boldsymbol\varepsilon(\mathbf v)\right)_{\Omega_+}+(p,\mathrm{div}\,\mathbf v)_{\Omega_+}\qquad \forall \mathbf v\in \mathbf W(\Omega_+).
\end{eqnarray*}
(Note how we have found it convenient to extend the exterior test space to the corresponding weighted Sobolev space.) We also have
\begin{equation}\label{eq:3.A4}
\|\mathbf t^\pm(\mathbf u,p)\|_{-1/2,\Gamma}\le C_\Gamma\Big( \|\boldsymbol\varepsilon(\mathbf u)\|_{\Omega_\pm}+\|p\|_{\Omega_\pm}\Big).
\end{equation}

\begin{proposition}\label{prop:4.2}\ 
\begin{itemize}
\item[{\rm (a)}] For smooth enough $(\mathbf u,p)$, $\mathbf t^\pm(\mathbf u,p)=2\nu\gamma^\pm(\boldsymbol\varepsilon(\mathbf u)) \,\mathbf n-(\gamma^\pm p)\mathbf n.$
\item[{\rm (b)}] For no interior velocity and constant pressure, the normal stress is proportional to the normal vector field:
$\mathbf t^-(\mathbf 0,1)=-\mathbf n$.
\item[{\rm (c)}] In dimensions $d=2$ and $d=3$, $\mathbf t^-(\mathbf u,p)\in \mathbf H^{-1/2}_0(\Gamma)$.
\item[{\rm (d)}] When $d=2$, $\mathbf t^+(\mathbf u,p)\in \mathbf H^{-1/2}_0(\Gamma)$.
\end{itemize}
\end{proposition}

\begin{proof}
Parts (a) and (b) follow from the definition. For part (c), take $\mathbf v\equiv \mathbf e \in \mathbb R^d$. For (d), do the same thing in the exterior domain.
\end{proof}

\section{Thinking of the boundary}

An issue with decompositions of vector fields on a Lipschitz boundary is related to the fact that the normal vector field is allowed to have discontinuities and therefore $\mathbf n\not\in \mathbf H^{1/2}(\Gamma)$. We still need to do some work on $\mathbf H^{\pm1/2}(\Gamma)$ related to the normal direction. What follows is one of several possible ways.
Let $\mathbf c_\Gamma:=\frac1{|\Gamma|}\int_\Gamma \mathbf x$ be the center of mass of $\Gamma$ and let  $\mathbf m(\mathbf x):=\mathbf x-\mathbf c_\Gamma$. Note that $\mathbf m\in \mathbf H^{1/2}(\Gamma)$, 
\begin{equation}\label{eq:3.A1}
\int_\Gamma \mathbf m\cdot\mathbf n =\int_{\Omega_-} \mathrm{div}(\mathbf x-\mathbf c_\Gamma) = d|\Omega_-|, \quad\mbox{and}\quad
\int_\Gamma \mathbf e\cdot\mathbf m =0 \quad \forall \mathbf e\in \boldsymbol P_0(\Gamma).
\end{equation}

\begin{proposition}\label{prop:3.1}
The decomposition
\[
\mathbf H^{1/2}(\Gamma)=\mathbf H^{1/2}_n(\Gamma)\oplus \mathrm{span}\,\{\mathbf m\}, \quad\mbox{with}\quad \mathbf H^{1/2}_n (\Gamma):=\{\boldsymbol\xi\in \mathbf H^{1/2}(\Gamma)\,:\,\int_\Gamma\boldsymbol\xi\cdot\mathbf n=0\},
\]
is stable. Moreover, $\boldsymbol P_0(\Gamma)\subset \mathbf H^{1/2}_n(\Gamma).$
\end{proposition}

\begin{proof} Given $\boldsymbol\xi\in \mathbf H^{1/2}(\Gamma)$, we can decompose
\[
\boldsymbol\xi=\boldsymbol\xi_0+c\mathbf m, \qquad c:=\frac1{d|\Omega_-|}\int_\Gamma \boldsymbol\xi\cdot\mathbf n=
\frac1{\int_\Gamma \mathbf m\cdot\mathbf n} \int_\Gamma \boldsymbol\xi\cdot\mathbf n, \qquad \boldsymbol\xi_0\in \mathbf H^{1/2}_n(\Gamma).
\]
Since $|c|\le |\Gamma|d^{-1}|\Omega_-|^{-1} \, \|\boldsymbol\xi\|_{\Gamma}$,
the stability of the decomposition follows readily. The inclusion $\boldsymbol P_0(\Gamma)\subset \mathbf H^{1/2}_n(\Gamma)$ is a simple consequence of the divergence theorem.
\end{proof}

\begin{proposition}\label{prop:3.2}
The decomposition
\[
\mathbf H^{-1/2}(\Gamma)=\mathbf H^{-1/2}_m(\Gamma)\oplus \mathrm{span}\,\{ \mathbf n\} , \quad\mbox{with}\quad \mathbf H^{-1/2}_m(\Gamma) :=\{\boldsymbol\lambda\in \mathbf H^{-1/2}(\Gamma)\,:\,\langle \boldsymbol\lambda,\mathbf m\rangle_\Gamma=0\},
\]
is stable. Moreover, $\boldsymbol P_0(\Gamma)\subset \mathbf H^{1/2}_m(\Gamma).$
\end{proposition}

\begin{proof}
Given $\boldsymbol\lambda\in \mathbf H^{-1/2}(\Gamma)$, we decompose
\[
\boldsymbol\lambda=\boldsymbol\lambda_0+ c\mathbf n, \qquad c=\frac{\langle \boldsymbol\lambda,\mathbf m\rangle_\Gamma}{\langle \mathbf n,\mathbf m\rangle_\Gamma}=\frac1{d|\Omega_-|}\langle\boldsymbol\lambda,\mathbf m\rangle_\Gamma \qquad \boldsymbol\lambda_0\in \mathbf H^{-1/2}_m(\Gamma).
\]
Since $|c| \le d^{-1}|\Omega_-|^{-1}\|\mathbf m\|_{1/2,\Gamma}\,\|\boldsymbol\lambda\|_{-1/2,\Gamma}$, 
the stability of the decomposition follows. The inclusion $\boldsymbol P_0(\Gamma)\subset \mathbf H^{1/2}_m(\Gamma)$ follows by \eqref{eq:3.A1}.
\end{proof}

Note that Proposition \ref{prop:3.2} is actually the dualization of Proposition \ref{prop:3.1}, given the fact that
\[
\mathbf H^{-1/2}_m(\Gamma) =\mathrm{span}\,\{ \mathbf m\}^\circ \qquad \mbox{and}\qquad \mathrm{span}\,\{ \mathbf n\}= \mathbf H^{1/2}_n(\Gamma)^\circ.
\]
We will also consider the following space:
\begin{equation}\label{eq:3.A3}
\mathbf H^{-1/2}_0(\Gamma) :=\{\boldsymbol\lambda\in \mathbf H^{-1/2}(\Gamma)\,:\,\langle \boldsymbol\lambda,\mathbf e\rangle_\Gamma=0\quad\forall \mathbf e \in \boldsymbol P_0(\Gamma)\}=\boldsymbol P_0(\Gamma)^\circ,
\end{equation}

\begin{proposition}\label{prop:3.3} 
The following decomposition is stable:
\[
\mathbf H^{-1/2}(\Gamma)=(\mathbf H^{-1/2}_m(\Gamma)\cap \mathbf H^{-1/2}_0(\Gamma)) \oplus\mathrm{span}\,\{ \mathbf n\} \oplus\boldsymbol P_0(\Gamma),
\]
\end{proposition}

\begin{proof} Given $\boldsymbol\lambda\in \mathbf H^{-1/2}(\Gamma)$, we decompose
$\boldsymbol\lambda=\boldsymbol\lambda_0+ c \mathbf n + \mathbf e$, 
where
\[
c:=\frac{\langle \boldsymbol\lambda,\mathbf m\rangle_\Gamma}{\langle \mathbf n,\mathbf m\rangle_\Gamma}\qquad \mbox{and}\qquad
\mathbf e:=\frac1{|\Gamma|}\sum_{j=1}^d \langle\boldsymbol\lambda,\mathbf e_j\rangle_\Gamma\mathbf e_j,
\]
$\{\mathbf e_1,\ldots,\mathbf e_d\}$ being the canonical basis of $\mathbb R^d$. This decomposition is stable and yields $\boldsymbol\lambda_0\in \mathbf H^{-1/2}_m(\Gamma)\cap \mathbf H^{-1/2}_0(\Gamma)$.
\end{proof}

\begin{proposition}[Traces of solenoidal fields]\label{prop:5.6}
The trace operator $\gamma :\mathbf V(\mathbb R^d):=\{ \mathbf u\in \mathbf W(\mathbb R^d)\,:\, \mathrm{div}\,\mathbf u=0\} \to \mathbf H^{1/2}_n(\Gamma)$
is surjective.
\end{proposition}

\begin{proof} If $\mathbf u \in \mathbf V(\mathbb R^d)$, then  $\gamma\mathbf u\in \mathbf H^{1/2}_n(\Gamma)$ by the divergence theorem.
 Let now $\boldsymbol\xi\in \mathbf H^{1/2}_n(\Gamma)$. We decompose $\mathbb R^d$ in the bounded set $\Omega_-$, surrounded by an annular domain $\Omega_b$ with boundary $\Gamma \cup\Sigma$ and an exterior unbounded domain $\Omega_e$ with boundary $\Sigma$. We then solve the following problems
\begin{alignat*}{8}
 -\Delta \mathbf u_-+\nabla p_- = \mathbf 0 & \qquad & \mbox{in $\Omega_-$}, & \qquad\qquad &
-\Delta \mathbf u_b+\nabla p_b = \mathbf 0 & \qquad & \mbox{in $\Omega_b$},\\
 \mathrm{div}\,\mathbf u_- = 0 & & \mbox{in $\Omega_-$}, & &
 \mathrm{div}\,\mathbf u_b = 0 & & \mbox{in $\Omega_b$},\\
 \gamma \mathbf u_-=\boldsymbol\xi, & & \mbox{on $\Gamma$}, & &
 \gamma \mathbf u_b=\boldsymbol\xi, & & \mbox{on $\Gamma$},\\
& & & &
\gamma \mathbf u_b=\mathbf 0, & & \mbox{on $\Sigma$},
\end{alignat*}
and define
\[
\mathbf u:= \left\{ \begin{array}{ll} \mathbf u_-, & \mbox{in $\Omega_-$},\\ \mathbf u_b , & \mbox{in $\Omega_b$},\\
\mathbf 0 & \mbox{in $\Omega_e$}.\end{array}\right.
\]
Then $\mathbf u\in \mathbf W(\mathbb R^d)$ and $\gamma\mathbf u=\boldsymbol\xi$.
\end{proof}

\section{Variational stokeslet distributions}\label{sec:5.1}

In what follows we will denote:
\[
a_{\mathcal O}(\mathbf u,\mathbf v):=2\nu \left( \boldsymbol\varepsilon(\mathbf u),\boldsymbol\varepsilon(\mathbf v)\right)_{\mathcal O}, \qquad 
\jump{\gamma\mathbf u}:=\gamma^-\mathbf u-\gamma^+\mathbf u.
\]
In this precise moment we start working only in the three dimensional case. Section \ref{sec:6.1} will also be focused in three dimensional space, while we will leave the modifications for two dimensions to Section \ref{sec:2d}. 
Given $\boldsymbol\lambda\in \mathbf H^{-1/2}(\Gamma)$, we look for the solution of
\begin{subequations}\label{eq:5.1}
\begin{alignat}{4}
(\mathbf u_\lambda,p_\lambda)\in \mathbf W(\mathbb R^3\setminus\Gamma)\times L^2(\mathbb R^3) &\\
 -2\nu \mathrm{div}\,\boldsymbol\varepsilon(\mathbf u_\lambda)+\nabla p_\lambda = \mathbf 0 & \qquad & \mbox{in $\mathbb R^3\setminus\Gamma$},\\
 \mathrm{div}\,\mathbf u_\lambda = 0 & & \mbox{in $\mathbb R^3\setminus\Gamma$},\\
 \jump{\gamma\mathbf u_\lambda}=0, & & \\
 \jump{\mathbf t(\mathbf u_\lambda,p_\lambda)}= \boldsymbol\lambda. & &
\end{alignat}
\end{subequations}

\begin{proposition}\label{prop:5.1}
Problem \eqref{eq:5.1} is equivalent to
\begin{equation}\label{eq:5.2}
\left[ \begin{array}{l}
\mathbf u_\lambda\in \mathbf W(\mathbb R^3), p_\lambda\in L^2(\mathbb R^3),\\[1.5ex]
\begin{array}{rll} \ds a_{\mathbb R^3}(\mathbf u_\lambda,\mathbf v)-(p_\lambda,\mathrm{div}\,\mathbf v)_{\mathbb R^3} &=\langle\boldsymbol\lambda,\gamma\mathbf v\rangle_\Gamma & \forall \mathbf v\in \mathbf W(\mathbb R^3),\\[1.5ex]
(\mathrm{div}\,\mathbf u_\lambda,q)_{\mathbb R^3} &=0 & \forall q \in L^2(\mathbb R^3).
\end{array}
\end{array} \right.
\end{equation}
Problem \eqref{eq:5.2} is well posed.
\end{proposition}

\begin{proof} The equivalence of both problems  is a simple application of distribution theory. Note that $\mathrm{div}\,\mathbf u_\lambda=0$ in $\mathbb R^3$ (since $\mathbf u_\lambda$ does not jump across $\Gamma$) but $-2\nu\mathrm{div}\,\boldsymbol\varepsilon(\mathbf u_\lambda)+\nabla p_\lambda =\mathbf 0$ only in $\mathbb R^3\setminus\Gamma$. Well posedness follows from the theory of abstract mixed problems \cite{Brezzi:1974}, \cite[Theorem II.1.1]{BrFo:1991} as we next show. By Proposition \ref{cor:2.5}, $\mathrm{div}:\mathbf W(\mathbb R^3) \to L^2(\mathbb R^3)$ is onto. By \eqref{eq:2.1}, we have coercivity
\[
a_{\mathbb R^3}(\mathbf u,\mathbf u)=2\nu \|\boldsymbol\varepsilon(\mathbf u)\|_{\mathbb R^3}^2\ge \nu \|\mathrm D\mathbf u\|_{\mathbb R^3}^2,
\]
and the right hand side of the inequality is an equivalent norm in $\mathbf W(\mathbb R^3)$ by Proposition \ref{prop:2.1} (see also Corollary \ref{cor:2.2}), so the bilinear form in the diagonal is coercive in the entire space.
\end{proof}

\paragraph{Definition.} To $\boldsymbol\lambda\in \mathbf H^{-1/2}(\Gamma)$ we associate $(\mathbf u_\lambda,p_\lambda)$ the solution of \eqref{eq:5.1} and then define the {\bf  single layer potential}
\[
\mathrm S_u\boldsymbol\lambda := \mathbf u_\lambda, \qquad \mathrm S_p\boldsymbol\lambda:= p_\lambda
\]
and the associated operators
\[
\mathrm V\boldsymbol\lambda :=\gamma^\pm \mathbf u_\lambda, \qquad \mathrm K^t\boldsymbol\lambda:=\ave{\mathbf t(\mathbf u_\lambda,p_\lambda)}=\smallfrac12\big( \mathbf t^+(\mathbf u_\lambda,p_\lambda)+ \mathbf t^-(\mathbf u_\lambda,p_\lambda)\big),
\]
respectively called single layer operator and adjoint double layer operator (it will take a while to see why the name of the latter).

\begin{proposition}[Mapping properties]\label{prop:5.3}
The following operators are bounded: 
\begin{eqnarray*}
(\mathrm S_u,\mathrm S_p) &:& \mathbf H^{-1/2}(\Gamma) \to \mathbf W(\mathbb R^3)\times L^2(\mathbb R^3),\\
\mathrm V &:& \mathbf H^{-1/2}(\Gamma) \to \mathbf H^{1/2}(\Gamma),\\
\mathrm K^t &:& \mathbf H^{-1/2}(\Gamma) \to \mathbf H^{-1/2}(\Gamma).
\end{eqnarray*}
\end{proposition}

\begin{proof} It is a direct consequence of the well posedness of \eqref{eq:5.2}, the definitions of the operators and \eqref{eq:3.A4}.
\end{proof}

\begin{proposition}[Jump relations]\label{prop:5.4}
Let $\boldsymbol\lambda\in \mathbf H^{-1/2}(\Gamma)$ and let $\mathbf u_\lambda:=\mathrm S_u\boldsymbol\lambda$, $p_\lambda:=\mathrm S_p\boldsymbol\lambda$. Then
\[
\jump{\gamma\mathbf u_\lambda}=\mathbf 0, \qquad 
\jump{\mathbf t(\mathbf u_\lambda,p_\lambda)}=\boldsymbol\lambda, \qquad \mbox{and}\qquad 
\mathbf t^\pm (\mathbf u_\lambda,p_\lambda)=\mp \smallfrac12 \boldsymbol\lambda +\mathrm K^t\boldsymbol\lambda.
\]
\end{proposition}

\begin{proof} The jump properties are given by definition of the potential. The normal stresses of the potential can then be computed by adding and subtracting the equalities
\[
\smallfrac12\jump{\mathbf t(\mathbf u_\lambda,p_\lambda)}=\smallfrac12\boldsymbol\lambda \qquad \mbox{and}\qquad \ave{\mathbf t(\mathbf u_\lambda,p_\lambda)}=\mathrm K^t\boldsymbol\lambda.
\]
\end{proof}

\paragraph{A particular solution.} For $\boldsymbol\lambda:=\mathbf n$, it is simple to check that $\mathbf u_\lambda\equiv 
\mathbf 0$ and $p_\lambda:=-\chi_{\Omega_-}$ is the solution of \eqref{eq:5.2}. Therefore $\mathrm S_u \mathbf n\equiv \mathbf 0$, $ \mathrm S_p \mathbf n=-\chi_{\Omega_-}$, and by Proposition \ref{prop:4.2}(b)
\[
\mathrm V \mathbf n=\mathbf 0, \qquad \mathrm K^t\mathbf n=-\smallfrac12\mathbf t^-(\mathbf 0,1)=\smallfrac12\mathbf n.
\]
The last identity can also be written
$(\mathrm K^t-\smallfrac12\mathrm I)\mathbf n=\mathbf 0.$

\begin{proposition}[Symmetry of $\mathrm V$]\label{prop:5.5}
For all $\boldsymbol\lambda,\boldsymbol\mu\in \mathbf H^{-1/2}(\Gamma)$,
\begin{equation}\label{eq:5.3}
\langle \boldsymbol\mu,\mathrm V\boldsymbol\lambda\rangle_\Gamma = \langle \boldsymbol\lambda,\mathrm V\boldsymbol\mu\rangle_\Gamma, \qquad \langle \boldsymbol\lambda,\mathrm V\boldsymbol\lambda\rangle_\Gamma\ge 0.
\end{equation}
Also
\[
\mathrm{Ker}\,\mathrm V =\mathrm{span}\,\{ \mathbf n\} \qquad \mbox{and} \qquad \mathrm{Range}\, \mathrm V \subset \mathbf H^{1/2}_n(\Gamma).
\]
\end{proposition}

\begin{proof}
We associate $\boldsymbol\lambda\mapsto (\mathbf u_\lambda,p_\lambda)$ and $\boldsymbol\mu\mapsto (\mathbf u_\mu,p_\mu)$. Then
\[
\langle \boldsymbol\mu,\mathrm V\boldsymbol\lambda\rangle_\Gamma= \langle \boldsymbol\mu,\gamma \mathbf u_\lambda\rangle_\Gamma= a_{\mathbb R^3}(\mathbf u_\mu,\mathbf u_\lambda)-(p_\mu,\mathrm{div}\,\mathbf u_\lambda)_{\mathbb R^3}=a_{\mathbb R^3}(\mathbf u_\mu,\mathbf u_\lambda).
\]
This proves \eqref{eq:5.3}.
Since $\mathrm  V\mathbf n=\mathbf 0$ (see above), it follows that
\[
0 =\langle\boldsymbol\lambda,\mathrm V\mathbf n\rangle_\Gamma=\langle \mathbf n,\mathrm V\boldsymbol\lambda\rangle_\Gamma,
\]
which, by definition, means that $\mathrm V\boldsymbol\lambda\in \mathbf H^{1/2}_n(\Gamma)$.

If $\mathrm V\boldsymbol\lambda=\mathbf 0$, then $a_{\mathbb R^3}(\mathbf u_\lambda,\mathbf u_\lambda)=0$. By Corollary \ref{cor:2.2}, this shows that $\mathbf u_\lambda\equiv \mathbf 0$. Going to the Stokes equations, this shows that $\nabla p_\lambda =\mathbf 0$ in $\mathbb R^3\setminus\Gamma$. Since $p_\lambda\in L^2(\mathbb R^3)$, this implies that $p_\lambda\in \mathrm{span}\{\chi_{\Omega_-}\}$. Using the fact that $\boldsymbol\lambda=\jump{\mathbf t(\mathbf u_\lambda,p_\lambda)}$ it follows that $\boldsymbol\lambda=c\jump{\mathbf t(\mathbf 0,\chi_{\Omega_-})}=-c\mathbf n$ for some $c\in \mathbb R$.
\end{proof}

\begin{proposition}[Coercivity of $\mathrm V$]\label{prop:5.7}
There exists $C_\Gamma>0$ such that
\[
\langle \boldsymbol\lambda,\mathrm V\boldsymbol\lambda\rangle_\Gamma\ge C_\Gamma \|\boldsymbol\lambda\|_{-1/2,\Gamma}^2 \qquad \forall \boldsymbol\lambda\in \mathbf H^{-1/2}_m(\Gamma).
\]
Therefore $\mathrm{Range}\,\mathrm V=\mathbf H^{1/2}_n(\Gamma)$.
\end{proposition}

\begin{proof}
 Recall that by Proposition \ref{prop:3.2}, the decomposition $\mathbf H^{1/2}(\Gamma)=\mathbf H^{1/2}_n(\Gamma)\oplus \mathrm{span}\,\{\mathbf m\}$ is stable. If $\boldsymbol\lambda\in \mathbf H^{-1/2}_m(\Gamma)$, this implies that
\[
\|\boldsymbol\lambda\|_{-1/2,\Gamma}=\sup_{\mathbf 0\neq \boldsymbol\xi\in \mathbf H^{1/2}_n(\Gamma)} \frac{|\langle \boldsymbol\lambda,\boldsymbol\xi\rangle_\Gamma|}{\|\boldsymbol\xi\|_{1/2,\Gamma}}.
\]
Using Proposition \ref{prop:5.6} we can define a right-inverse of the trace operator $\gamma^\dagger: \mathbf H^{1/2}_n(\Gamma) \to \mathbf V(\mathbb R^3)$. Then
\begin{eqnarray*}
|\langle \boldsymbol\lambda,\boldsymbol\xi\rangle_\Gamma| &=& |\langle \boldsymbol\lambda, \gamma \gamma^\dagger\boldsymbol\xi\rangle_\Gamma|=|a_{\mathbb R^3}(\mathbf u_\lambda,\gamma^\dagger\boldsymbol\xi)-(p_\lambda,\mathrm{div}\gamma^\dagger\boldsymbol\xi)_{\mathbb R^3}|\\
&=& |a_{\mathbb R^3}(\mathbf u_\lambda,\gamma^\dagger\boldsymbol\xi)| \le \nu C_\Gamma \|\boldsymbol\varepsilon(\mathbf u_\lambda)\|_{\mathbb R^3}\|\boldsymbol\xi\|_{1/2,\Gamma},
\end{eqnarray*}
and therefore
\[
\|\boldsymbol\lambda\|_{-1/2,\Gamma} \le \nu C_\Gamma \|\boldsymbol\varepsilon(\mathbf u_\lambda)\|_{\mathbb R^3} \qquad \forall \boldsymbol\lambda \in \mathbf H^{-1/2}_m(\Gamma), \quad \mathbf u_\lambda:=\mathrm S_u\boldsymbol\lambda.
\]
Finally from the proof of Proposition \ref{prop:5.5} it follows that
\[
\langle\boldsymbol\lambda,\mathrm V\boldsymbol\lambda\rangle_\Gamma=a_{\mathbb R^3} (\mathbf u_\lambda,\mathbf u_\lambda)\ge \frac2{\nu C_\Gamma^2}\|\boldsymbol\lambda\|_{-1/2,\Gamma}^2,
\]
which proves the coercivity of $\mathrm V$. Finally, if $\boldsymbol\xi\in \mathbf H^{1/2}_n(\Gamma)$, then we can solve the coercive problem
\[
\boldsymbol\lambda\in \mathbf H^{-1/2}_m(\Gamma) \qquad \langle \boldsymbol\mu,\mathrm V\boldsymbol\lambda\rangle_\Gamma =\langle \boldsymbol\mu,\boldsymbol\xi\rangle_\Gamma \quad \forall \boldsymbol\mu \in \mathbf H^{-1/2}_m(\Gamma)
\]
and note that by Proposition \ref{prop:3.2} this implies that $\mathrm V \boldsymbol\lambda =\boldsymbol\xi$.
\end{proof}

\section{Variational stresslet distributions}\label{sec:6.1}

In parallel with Section \ref{sec:5.1}, we restrict our attention to the three dimensional case. 
Given $\boldsymbol\varphi\in \mathbf H^{1/2}(\Gamma)$ we look for solutions of
\begin{subequations}\label{eq:6.1}
\begin{alignat}{4}
(\mathbf u_\varphi,p_\varphi)\in \mathbf W(\mathbb R^3\setminus\Gamma)\times L^2(\mathbb R^3) \\\
 -2\nu \mathrm{div}\,\boldsymbol\varepsilon(\mathbf u_\varphi)+\nabla p_\varphi = \mathbf 0 & \qquad & \mbox{in $\mathbb R^3\setminus\Gamma$},\\
 \mathrm{div}\,\mathbf u_\varphi = 0 & & \mbox{in $\mathbb R^3\setminus\Gamma$},\\
 \jump{\gamma\mathbf u_\varphi}=-\boldsymbol\varphi, & & \\
 \jump{\mathbf t(\mathbf u_\varphi,p_\varphi)}= \mathbf 0. & &
\end{alignat}
\end{subequations}

\begin{proposition}\label{prop:6.1}
Problem \eqref{eq:6.1} is equivalent to
\begin{equation}\label{eq:6.2}
\left[ \begin{array}{l}
\mathbf u_\varphi\in \mathbf W(\mathbb R^3\setminus\Gamma), p_\varphi\in L^2(\mathbb R^3),\\[1.5ex]
\jump{\gamma \mathbf u_\varphi}+\boldsymbol\varphi=\mathbf 0,\\[1.5ex]
\begin{array}{rll} \ds a_{\mathbb R^3\setminus\Gamma}(\mathbf u_\varphi,\mathbf v)-(p_\varphi,\mathrm{div}\,\mathbf v)_{\mathbb R^3} &=\mathbf 0 & \forall \mathbf v\in \mathbf W(\mathbb R^3),\\[1.5ex]
(\mathrm{div}\,\mathbf u_\varphi,q)_{\mathbb R^3\setminus\Gamma} &=0 & \forall q \in L^2(\mathbb R^3).
\end{array}
\end{array} \right.
\end{equation}
Problem \eqref{eq:6.2} is well posed.
\end{proposition}

\begin{proof} Equivalence of \eqref{eq:6.1} and \eqref{eq:6.2}  is a simple application of distribution theory. 
Taking $L\boldsymbol\varphi\in \mathbf W(\mathbb R^3)$ such that $\jump{\gamma L\boldsymbol\varphi}=-\boldsymbol\varphi$, we can rewrite the problem in terms of the variable $\mathbf u_\varphi-L\boldsymbol\varphi\in \mathbf W(\mathbb R^3)$. The resulting problem is associated to the same bilinear form as problem \eqref{eq:5.2} and is therefore well posed.
\end{proof}

\paragraph{Definition.} To $\boldsymbol\varphi\in \mathbf H^{1/2}(\Gamma)$ we associate $(\mathbf u_\varphi,p_\varphi)$ the solution of \eqref{eq:6.1} and then define the {\bf  double layer potential}
\[
\mathrm D_u\boldsymbol\varphi := \mathbf u_\varphi, \qquad \mathrm D_p\boldsymbol\varphi:= p_\varphi
\]
and the associated operators
\[
\mathrm W\boldsymbol\varphi :=-\mathbf t^\pm(\mathbf u_\varphi,p_\varphi), \qquad \mathrm K\boldsymbol\varphi:=\ave{\gamma\mathbf u_\varphi}=\smallfrac12\big( \gamma^+\mathbf u_\varphi+ \gamma^-\mathbf u_\varphi\big),
\]
respectively named hypersingular operator and double layer operator.

\begin{proposition}[Mapping properties]\label{prop:6.3}
The following operators are bounded:
\begin{eqnarray*}
(\mathrm D_u,\mathrm D_p) &:& \mathbf H^{1/2}(\Gamma) \to \mathbf W(\mathbb R^3\setminus\Gamma)\times L^2(\mathbb R^3),\\
\mathrm W &:& \mathbf H^{1/2}(\Gamma) \to \mathbf H^{-1/2}(\Gamma),\\
\mathrm K &:& \mathbf H^{1/2}(\Gamma) \to \mathbf H^{1/2}(\Gamma).
\end{eqnarray*}
\end{proposition}

\begin{proof} It is a direct consequence of the well posedness of \eqref{eq:6.2}, the definitions of the operators and \eqref{eq:3.A4}.
\end{proof}

\begin{proposition}[Jump relations]\label{prop:6.4}
Let $\boldsymbol\varphi\in \mathbf H^{1/2}(\Gamma)$ and let $\mathbf u_\varphi:=\mathrm D_u\boldsymbol\varphi$, $p_\varphi:=\mathrm D_p\boldsymbol\varphi$. Then
\[
\jump{\gamma\mathbf u_\varphi}=-\boldsymbol\varphi, \qquad 
\jump{\mathbf t(\mathbf u_\varphi,p_\varphi)}=\mathbf 0, \qquad 
\gamma^\pm \mathbf u_\varphi=\pm \smallfrac12 \boldsymbol\varphi +\mathrm K\boldsymbol\varphi.
\]
\end{proposition}

\begin{proof} It is a straightforward consequence of the definitions.
\end{proof}

\paragraph{A kernel.} For $\mathbf c\in \mathcal M$, the fields $\mathbf u=-\chi_{\Omega_-}\mathbf c$ and $p=0$ solve \eqref{eq:6.1} with $\boldsymbol\varphi=\mathbf c\in \mathcal M_\Gamma=\gamma \mathcal M$. This leads to the identities:
\begin{equation}\label{eq:6.10}
\mathrm D_u\mathbf c=-\chi_{\Omega_-}\mathbf c , \qquad\mathrm D_p \mathbf c=0,\qquad \mathrm K\mathbf c=-\smallfrac12\mathbf c, \qquad \mathrm W\mathbf c=\mathbf 0, \qquad \forall \mathbf c\in \mathcal M_\Gamma.
\end{equation}

\begin{proposition}[Symmetry of $\mathrm W$] \label{prop:6.5}
For all $\boldsymbol\varphi,\boldsymbol\psi\in \mathbf H^{1/2}(\Gamma)$,
\begin{equation}\label{eq:6.3}
\langle \mathrm W\boldsymbol\varphi,\boldsymbol\psi\rangle_\Gamma = \langle \mathrm W\boldsymbol\psi,\boldsymbol\varphi\rangle_\Gamma, \qquad \langle  \mathrm W\boldsymbol\varphi, \boldsymbol\varphi\rangle_\Gamma\ge 0.
\end{equation}
Also
\[
\mathrm{Ker}\,\mathrm W =\mathcal M_\Gamma \qquad \mbox{and} \qquad \mathrm{Range}\,\mathrm W \subset \mathcal M_\Gamma^\circ.
\]
\end{proposition}

\begin{proof}
Let $(\mathbf u_\varphi,p_\varphi):=(\mathrm D_u\boldsymbol\varphi,\mathrm D_p\boldsymbol\varphi)$ and $\mathbf u_\psi:=\mathrm D_u\boldsymbol\psi$. Then
\begin{eqnarray*}
\langle\mathrm W\boldsymbol\varphi,\boldsymbol\psi\rangle_\Gamma &=& \langle \mathbf t(\mathbf u_\varphi,p_\varphi),\jump{\gamma\mathbf u_\psi}\rangle_\Gamma =\langle \mathbf t^-(\mathbf u_\varphi,p_\varphi),\gamma^-\mathbf u_\psi\rangle_\Gamma-\langle \mathbf t^+(\mathbf u_\varphi,p_\varphi),\gamma^+\mathbf u_\psi\rangle_\Gamma\\
&=& a_{\mathbb R^3\setminus\Gamma}(\mathbf u_\varphi,\mathbf u_\psi).
\end{eqnarray*} 
This equality implies symmetry and positive semi-definiteness. 

We already know from \eqref{eq:6.10} that $\mathcal M_\Gamma\subset \mathrm{Ker}\,\mathrm W$.
If $\mathrm W\boldsymbol\varphi=\mathbf 0$, then, by the previous identity, it follows that $\boldsymbol\varepsilon(\mathbf u_\varphi)=0$ in $\mathbb R^3\setminus\Gamma$. By Lemma \ref{lemma:rigid}, it follows that $\mathbf u_\varphi|_{\Omega_-}\in \mathcal M$ and $\mathbf u|_{\Omega_+}\equiv \mathbf 0$. Therefore $\boldsymbol\varphi=-\jump{\gamma\mathbf u_\varphi}=-\gamma^-\mathbf u\in \mathcal M_\Gamma$. 

Finally, if $\boldsymbol\varphi\in \mathbf H^{1/2}(\Gamma)$ and $\mathbf c\in \mathcal M$, then defining $(\mathbf u_\varphi,p_\varphi)$ as usual,
\[
\langle\mathrm W\boldsymbol\varphi,\mathbf c\rangle_\Gamma=-\langle \mathbf t^-(\mathbf u_\varphi,p_\varphi),\gamma \mathbf c\rangle_\Gamma=-2\nu (\boldsymbol\varepsilon(\mathbf u_\varphi),\boldsymbol\varepsilon(\mathbf c))_{\Omega_-}+(p_\varphi,\mathrm{div}\mathbf c)_{\Omega_-}=0,
\]
which proves that $ \mathrm W \boldsymbol\varphi\in \mathcal M_\Gamma^\circ$.
\end{proof}

\begin{proposition}[Coercivity of $\mathrm W$]\label{prop:6.6}
There exists $C_\Gamma>0$ such that
\[
\langle \mathrm W\boldsymbol\varphi,\boldsymbol\varphi\rangle_\Gamma \ge C_\Gamma\|\boldsymbol\varphi\|_{1/2,\Gamma}^2 \qquad \forall \boldsymbol\varphi\in \mathbf H^{1/2}_{\mathcal M}(\Gamma):=\{\boldsymbol\xi\,:\,\int_\Gamma \boldsymbol\xi\cdot\mathbf c=0\quad\forall \mathbf c\in \mathcal M_\Gamma\}.
\]
Therefore $\mathrm{Range}\,\mathrm W=\mathcal M_\Gamma^\circ$.
\end{proposition}

\begin{proof}
Proceeding as before, with $\mathbf u_\varphi=\mathrm D\boldsymbol\varphi$, we prove that
\[
\langle\mathrm W\boldsymbol\varphi,\boldsymbol\varphi\rangle_\Gamma = 2\nu\|\boldsymbol\varepsilon(\mathbf u_\varphi)\|_{\mathbb R^3\setminus\Gamma}^2.
\]
Using Proposition \ref{prop:2.8}, Korn's inequality and a compactness argument, it is easy to prove that
\[
|\!|\!|\mathbf u|\!|\!|^2:=\|\boldsymbol\varepsilon(\mathbf u)\|_{\mathbb R^3\setminus\Gamma}+\sum_{\ell=1}^6 \Big| \int_\Gamma \jump{\gamma\mathbf u}\cdot\mathbf c_\ell\Big|^2\qquad \mbox{where } \mathcal M=\mathrm{span}\{\mathbf c_1,\ldots,\mathbf c_6\},
\]
is equivalent to the usual norm in $\mathbf W(\mathbb R^3\setminus\Gamma)$. However, if $\boldsymbol\varphi=-\jump{\gamma\mathbf u}\in \mathbf H^{1/2}_{\mathcal M}(\Gamma)$, then $|\!|\!|\mathbf u|\!|\!|=\|\boldsymbol\varepsilon(\mathbf u)\|_{\mathbb R^3\setminus\Gamma}$ and thus
\[
\langle\mathrm W\boldsymbol\varphi,\boldsymbol\varphi\rangle_\Gamma  \ge C_\Gamma \|\mathbf u_\varphi\|^2_{1,\rho,\mathbb R^3\setminus\Gamma}.
\]
Finally
\[
\|\boldsymbol\varphi\|_{1/2,\Gamma}\le \|\gamma^+\mathbf u_\varphi\|_{1/2,\Gamma}+\|\gamma^-\mathbf u_\varphi\|_{1/2,\Gamma} \le C \|\mathbf u_\varphi\|_{1,\rho,\mathbb R^3\setminus\Gamma}
\]
and the proof of the coercivity estimate is finished. To show that the range of $\mathrm W$ is $\mathcal M_\Gamma^\circ$, follow the steps of the end of the proof of Proposition \ref{prop:5.7}, using now the decomposition $\mathbf H^{1/2}(\Gamma)=\mathbf H^{1/2}_{\mathcal M}(\Gamma)\oplus \mathcal M_\Gamma$.
\end{proof}

\begin{proposition}\label{prop:6.7}
Single and double layer potentials are orthogonal with respect to the semi-inner product $a_{\mathbb R^3\setminus\Gamma}(\punto,\punto)$, that is,
\[
\mathbf v=\mathrm S_u\boldsymbol\lambda, \quad \mathbf u=\mathrm D_u\boldsymbol\varphi \quad \Longrightarrow\quad
a_{\mathbb R^3\setminus\Gamma}(\mathbf u,\mathbf v)=0.
\]
\end{proposition}

\begin{proof} Taking $\mathbf v=\mathrm S_u\boldsymbol\lambda\in \mathbf W(\mathbb R^3)$ as test function in problem \eqref{eq:6.2} the result follows.
\end{proof}

\begin{proposition}[$\mathrm K^t$ is the transpose of $\mathrm K$]\label{prop:6.8}
\[
\langle \boldsymbol\lambda,\mathrm K\boldsymbol\varphi\rangle_\Gamma= \langle \mathrm K^t\boldsymbol\lambda,\boldsymbol\varphi\rangle_\Gamma \qquad \forall \boldsymbol\lambda\in \mathbf H^{-1/2}(\Gamma), \quad \boldsymbol\varphi\in \mathbf H^{1/2}(\Gamma).
\]
\end{proposition}

\begin{proof}
Let $(\mathbf u_\lambda,p_\lambda):=(\mathrm S_u\boldsymbol\lambda,\mathrm S_p\boldsymbol\lambda)$ and $(\mathbf u_\varphi,p_\varphi):=(\mathrm D_u\boldsymbol\varphi,\mathrm D_p\boldsymbol\varphi)$. By Proposition \ref{prop:6.7} it follows that
\begin{eqnarray*} \langle \mathbf t^-(\mathbf u_\lambda,p_\lambda),\gamma^-\mathbf u_\varphi\rangle_\Gamma &=& a_{\Omega_-}(\mathbf u_\lambda,\mathbf u_\varphi)-(p_\lambda,\mathrm{div}\,\mathbf u_\varphi)_{\Omega_-} =
 a_{\Omega_-}(\mathbf u_\lambda,\mathbf u_\varphi)\\
&=& -a_{\Omega_+} (\mathbf u_\lambda,\mathbf u_\varphi)=
-a_{\Omega_+} (\mathbf u_\lambda,\mathbf u_\varphi)
+(p_\lambda,\mathrm{div}\,\mathbf u_\varphi)_{\Omega_+}\\
&=&\langle \mathbf t^+(\mathbf u_\lambda,p_\lambda),\gamma^+\mathbf u_\varphi\rangle_\Gamma .
\end{eqnarray*}
However, this equality is equivalent
\[
\langle \mathbf t^+(\mathbf u_\lambda,p_\lambda),-\jump{\gamma\mathbf u_\varphi}\rangle_\Gamma =\langle \jump{\mathbf t(\mathbf u_\lambda,p_\lambda)},\gamma^-\mathbf u_\varphi\rangle_\Gamma,
\]
which, by the jump properties (Propositions \ref{prop:5.4} and \ref{prop:6.4}), can be equivalently written as
\[
\langle-\smallfrac12\boldsymbol\lambda+\mathrm K^t \boldsymbol\lambda,\boldsymbol\varphi\rangle_\Gamma = \langle\boldsymbol\lambda,-\smallfrac12\boldsymbol\varphi+\mathrm K\boldsymbol\varphi\rangle_\Gamma,
\]
which, once more, is equivalent to the statement. The reader who has followed the proof in detail will have realized that this transposition theorem ends up being equivalent to the orthogonality property (Proposition \ref{prop:6.7}), which asserts that energy of single and double layer potentials remains separated. 
\end{proof}

\begin{proposition}[Representation formula]\label{prop:7.2}
Let $\mathbf u\in \mathbf W(\mathbb R^3\setminus\Gamma)$ and $p\in L^2(\mathbb R^3)$ satisfy
\begin{alignat*}{4}
 -2\nu \mathrm{div}\,\boldsymbol\varepsilon(\mathbf u)+\nabla p = \mathbf 0 & \qquad & \mbox{in $\mathbb R^3\setminus\Gamma$},\\
 \mathrm{div}\,\mathbf u = 0 & & \mbox{in $\mathbb R^3\setminus\Gamma$}.
\end{alignat*}
Then
\begin{equation}
\mathbf u = \mathrm S_u \jump{\mathbf t(\mathbf u,p)}-\mathrm D_u \jump{\gamma\mathbf u}\quad\mbox{and}\quad
p = \mathrm S_p \jump{\mathbf t(\mathbf u,p)}-\mathrm D_p \jump{\gamma\mathbf u}.
\end{equation}
\end{proposition}

\begin{proof}
Let $\boldsymbol\lambda:=\jump{\mathbf t(\mathbf u,p)}$ and $\boldsymbol\xi:=\jump{\gamma\mathbf u}$. Define then $\mathbf v:=\mathrm S_u\boldsymbol\lambda-\mathrm D_u\boldsymbol\xi $ and $ q=\mathrm S_p\boldsymbol\lambda-\mathrm D_p\boldsymbol\xi$.
By definition of the layer potentials, this pair satisfies
\begin{alignat*}{4}
 -2\nu \,\mathrm{div}\,\boldsymbol\varepsilon(\mathbf v)+\nabla q = \mathbf 0 & \qquad & \mbox{in $\mathbb R^3\setminus\Gamma$},\\
 \mathrm{div}\,\mathbf v = 0 & & \mbox{in $\mathbb R^3\setminus\Gamma$},\\
 \jump{\gamma\mathbf v}=\boldsymbol\xi, & & \\
 \jump{\mathbf t(\mathbf v,q)}= \boldsymbol\lambda. & &
\end{alignat*}
However, $(\mathbf u,p)$ satisfies the same equations. This problem has at most one solution (this is part of Proposition \ref{prop:5.1} and \ref{prop:6.1}).
Therefore, $\mathbf u=\mathbf v$ and $p=q$, which proves the result.
\end{proof}

\section{Two dimensional hurdles}\label{sec:2d}

There are two new difficulties in the two dimensional case:
\begin{itemize}
\item The exterior normal stress has zero average on the boundary $\Gamma$ (see Proposition \ref{prop:4.2}(d)). This is also reflected in the fact that the fundamental solution does not decay and a condition of zero average on the density is needed to ensure decay of velocity and pressure at infinity.
\item On the other hand, the associated weighted Sobolev space contains constant functions --Proposition \ref{prop:2.1}, Corollary \ref{cor:2.2} and formula \eqref{eq:2.A1}--. This fact leaks a two dimensional kernel (constant velocity fields with no pressure) in the set of solutions in free space. The associated transmission problems that we use to define the layer potentials are solvable up to constant velocity fields. Since the integral representations of the potentials show decaying velocity fields, we have to choose the right one from the variational formulation. The way we are going to do it is by looking around with the integral formulation (this is a {\em Deus ex machina} moment in the theory), picking up an average and imposing it on the solution. There might be simpler ways. (For instance, in the Laplace equation, the average around can be shown to be zero, which is how the right solution is chosen.)
\end{itemize}
As a byproduct of all of this, the representation formula shows that the jumps of the Cauchy data represent the solution up to a constant velocity field. This fact ends up being the variational form of the Stokes paradox.

\paragraph*{That Fredholm feeling.} If we try to define the layer potentials with the transmission problems
\begin{subequations}\label{eq:9.1}
\begin{alignat}{4}
(\mathbf u,p)\in \mathbf W(\mathbb R^2\setminus\Gamma)\times L^2(\mathbb R^2),\\
 -2\nu \,\mathrm{div}\,\boldsymbol\varepsilon(\mathbf u)+\nabla p = \mathbf 0 & \qquad & \mbox{in $\mathbb R^2\setminus\Gamma$},\\
 \mathrm{div}\,\mathbf u = 0 & & \mbox{in $\mathbb R^2\setminus\Gamma$},\\
 \jump{\gamma\mathbf u}=\boldsymbol\varphi, & & \\
 \jump{\mathbf t(\mathbf u,p)}= \boldsymbol\lambda, & &
\end{alignat}
\end{subequations}
we will be instantly annoyed by the existence of a two dimensional kernel: $(\mathbf u_\infty,0)\in \boldsymbol P_0(\mathbb R^2)\times \{0\}$ is, unfortunately, a homogeneous solution of \eqref{eq:9.1}. Anyone slightly aquainted with Fredholm properties will be ready to see that we need to ground the solution (in order to find a unique solution), and that the data $(\boldsymbol\varphi,\boldsymbol\lambda)\in \mathbf H^{1/2}(\Gamma)\times \mathbf H^{-1/2}(\Gamma)$ will be forced to satisfy a compatibility condition. The second requirement is easy: it falls entirely on the demand that $\boldsymbol\lambda\in \mathbf H^{-1/2}_0(\Gamma)$, which is necessary for existence of solution by Proposition \ref{prop:4.2}(c)-(d). We will ground the solution by integrating around the domain.
We thus define
\[
\boldsymbol\jmath(\mathbf u):=\int_\Xi \mathbf u(\mathbf x)\mathrm d \Xi(\mathbf x),
\]
where $\Xi=\partial B(\mathbf 0;R)$ and $\overline{\Omega^-}\subset B(\mathbf 0;R)$. The value for this surrounding integral will be taken from some moments on the two densities. This is where we bring some information from the future (from what it will be when we `go all integral'). With the matrix valued kernels (see ahead in \eqref{eq:11.31})
\[
\mathrm E_u(\mathbf r):=\frac1{4\pi\nu} \left( \log r^{-1}\,\mathrm I +\frac1{r^2} \mathbf r\otimes\mathbf r\right), \qquad
\mathrm T_u(\mathbf r;\mathbf n):=\frac1{\pi r^4}(\mathbf r\cdot\mathbf n)(\mathbf r \otimes\mathbf r),
\]
we construct two matrix-valued functions on $\Gamma$ (generated from $\Xi$)
\begin{equation}\label{eq:9.50}
\mathrm B_{\mathrm S}(\mathbf y):=\int_\Xi \mathrm E_u(\mathbf x-\mathbf y)\mathrm d\Xi(\mathbf x),\qquad 
 \mathrm B_{\mathrm D}(\mathbf y):=\int_\Xi \mathrm T_u(\mathbf x-\mathbf y;\mathbf n(\mathbf y))\,\mathrm d \Xi(\mathbf x),\qquad \mathbf y\in \Gamma,
\end{equation}
and then two $\mathbb R^2$-valued linear functionals
\begin{equation}\label{eq:9.51}
\boldsymbol\ell_{\mathrm S} (\boldsymbol\lambda):=\langle \mathrm B_{\mathrm S},\boldsymbol\lambda\rangle_\Gamma,
\qquad 
\boldsymbol\ell_{\mathrm D}(\boldsymbol\varphi):=\int_\Gamma \mathrm B_{\mathrm D}(\mathbf y)\boldsymbol\varphi(\mathbf y)\,\mathrm d\Gamma(\mathbf y).
\end{equation}
The grounding condition for problem \eqref{eq:9.1} will then be
\begin{equation}\label{eq:9.52}
\boldsymbol\jmath(\mathbf u)=\boldsymbol\ell_{\mathrm S}(\boldsymbol\lambda)-\boldsymbol\ell_{\mathrm D}(\boldsymbol\varphi),
\end{equation}
and the space
\begin{equation}\label{eq:9.20}
\mathbf W_\star(\mathbb R^2):=\{\mathbf u\in \mathbf W(\mathbb R^2)\,:\,\boldsymbol\jmath(\mathbf u)=\mathbf 0\}
\end{equation}
will play the role of the ground zero. From this moment on, the narrative is going to be extremely (while not strikingly) similar to that of Sections \ref{sec:5.1} 
and \ref{sec:6.1}. Since at this point the reader will be used to these arguments, we are going to be bolder and to define both layer potentials at the same time. Most proofs of this section are straightforward copies of those in Sections \ref{sec:5.1} and \ref{sec:6.1}.

\begin{proposition}\label{prop:AA}
Problem \eqref{eq:9.1} is equivalent to
\begin{equation}\label{eq:9.2}
\left[ \begin{array}{l}
\mathbf u\in \mathbf W(\mathbb R^2\setminus\Gamma), p\in L^2(\mathbb R^2),\\[1.5ex]
\jump{\gamma \mathbf u}=\boldsymbol\varphi,\\[1.5ex]
\begin{array}{rll} \ds a_{\mathbb R^2\setminus\Gamma}(\mathbf u,\mathbf v)-(p,\mathrm{div}\,\mathbf v)_{\mathbb R^2} &=\langle\boldsymbol\lambda,\gamma\mathbf v\rangle_\Gamma & \forall \mathbf v\in \mathbf W(\mathbb R^2),\\[1.5ex]
(\mathrm{div}\,\mathbf u,q)_{\mathbb R^2} &=0 & \forall q \in L^2(\mathbb R^2).
\end{array}
\end{array} \right.
\end{equation}
The modified problem 
\begin{equation}\label{eq:9.3}
\left[ \begin{array}{l}
\mathbf u\in \mathbf W(\mathbb R^2\setminus\Gamma), p\in L^2(\mathbb R^2),\\[1.5ex]
\jump{\gamma \mathbf u}=\boldsymbol\varphi,\\[1.5ex]
\boldsymbol\jmath(\mathbf u)=\boldsymbol\ell_{\mathrm S}(\boldsymbol\lambda)-\boldsymbol\ell_{\mathrm D}(\boldsymbol\varphi),\\[1.5ex]
\begin{array}{rll} \ds a_{\mathbb R^2\setminus\Gamma}(\mathbf u,\mathbf v)-(p,\mathrm{div}\,\mathbf v)_{\mathbb R^2} &=\langle\boldsymbol\lambda,\gamma\mathbf v\rangle_\Gamma & \forall \mathbf v\in \mathbf W(\mathbb R^2),\\[1.5ex]
(\mathrm{div}\,\mathbf u,q)_{\mathbb R^2} &=0 & \forall q \in L^2(\mathbb R^2).
\end{array}
\end{array} \right.
\end{equation}
is well posed. Therefore, problem \eqref{eq:9.2} has a solution that is unique up to elements of $\boldsymbol P_0(\mathbb R^2)\times \{0\}$.
\end{proposition}

\begin{proof}
The equivalence of \eqref{eq:9.1} and \eqref{eq:9.2} is, as usual, a simple exercise in distribution theory. Note now that
\begin{equation}\label{eq:9.22}
\mathbf W(\mathbb R^2)=\mathbf W_\star(\mathbb R^2)\oplus \boldsymbol P_0(\mathbb R^2).
\end{equation}
Since $\boldsymbol\lambda\in \mathbf H^{-1/2}_0(\Gamma)$, then the decomposition \eqref{eq:9.22} shows that
\[
a_{\mathbb R^2}(\mathbf u_\lambda,\mathbf v)-(p_\lambda,\mathrm{div}\,\mathbf v)_{\mathbb R^2} =\langle\boldsymbol\lambda,\gamma\mathbf v\rangle_\Gamma \qquad \forall \mathbf v\in \mathbf W_\star(\mathbb R^2)
\]
if and only if
\[
a_{\mathbb R^2}(\mathbf u_\lambda,\mathbf v)-(p_\lambda,\mathrm{div}\,\mathbf v)_{\mathbb R^2} =\langle\boldsymbol\lambda,\gamma\mathbf v\rangle_\Gamma \qquad \forall \mathbf v\in \mathbf W(\mathbb R^2).
\]
We next get rid of the non-homogeneous side conditions. We start constructing
\[
\mathbf u_\varphi^\circ\in \mathbf W(\mathbb R^d\setminus\Gamma), \qquad \gamma^-\mathbf u_\varphi^\circ=\boldsymbol\varphi, \qquad \mathbf u_\varphi^\circ\equiv 0\mbox{ in $\Omega_+$},
\]
so that $\jump{\gamma\mathbf u_\varphi^\circ}=\boldsymbol\varphi$ and $\boldsymbol\jmath(\mathbf u_\varphi^\circ)=\mathbf 0$. We then propose
\[
\mathbf u=\mathbf u^\star+\mathbf u_\varphi^\circ+\smallfrac1{|\Xi|} \big(\boldsymbol\ell_{\mathrm S}(\boldsymbol\lambda)-\boldsymbol\ell_{\mathrm D}(\boldsymbol\varphi)\big),
\]
where
\begin{equation}\label{eq:9.4}
\left[ \begin{array}{l}
\mathbf u^\star\in \mathbf W_\star(\mathbb R^2), p\in L^2(\mathbb R^2),\\[1.5ex]
\begin{array}{rll} \ds a_{\mathbb R^2}(\mathbf u^\star,\mathbf v)-(p,\mathrm{div}\,\mathbf v)_{\mathbb R^2} &=\langle\boldsymbol\lambda,\gamma\mathbf v\rangle_\Gamma-a_{\Omega_-}(\mathbf u_\varphi^\circ,\mathbf v) & \forall \mathbf v\in \mathbf W_\star(\mathbb R^2),\\[1.5ex]
(\mathrm{div}\,\mathbf u^\star,q)_{\mathbb R^2} &=0 & \forall q \in L^2(\mathbb R^2),
\end{array}
\end{array} \right.
\end{equation}
as the solution of \eqref{eq:9.3}. That problem \eqref{eq:9.4} is well posed follows from the theory of mixed problems and the arguments we next give. 
By  Corollary \ref{cor:2.2}, $\|\boldsymbol\varepsilon(\punto)\|_{\mathbb R^2}$ is equivalent to the weighted Sobolev norm in $\mathbf W_\star(\mathbb R^2)$, which means that the diagonal bilinear form $a_{\mathbb R^2}$ is coercive in $\mathbf W_\star(\mathbb R^2)$.  By Proposition \ref{cor:2.5}, $\mathrm{div}:\mathbf W(\mathbb R^2) \to L^2(\mathbb R^2)$ is onto and therefore, using \eqref{eq:9.22}, it follows that $\mathrm{div}:\mathbf W_\star(\mathbb R^2) \to L^2(\mathbb R^2)$ is also surjective. 
\end{proof}

\paragraph{Two (or four) potentials, four operators.} Problem \eqref{eq:9.3} defines a bounded linear map $\mathbf H^{1/2}(\Gamma)\times \mathbf H^{-1/2}_0(\Gamma) \to \mathbf W(\mathbb R^2\setminus\Gamma)\times L^2(\mathbb R^2)$. We are then allowed to see this as a matrix of operators
\[
\left[\begin{array}{c}\mathbf u \\ p\end{array}\right]=\left[\begin{array}{cc} \mathrm S_u & -\mathrm D_u \\ \mathrm S_p & -\mathrm D_p\end{array}\right]\left[\begin{array}{c} \boldsymbol\lambda \\ \boldsymbol\varphi\end{array}\right]=\left[ \begin{array}{cc} \mathbb S & -\mathbb D\end{array}\right]\left[\begin{array}{c} \boldsymbol\lambda \\ \boldsymbol\varphi\end{array}\right],
\]
where the second expression looks at $(\mathbf u,p)$ as a joint entity, and not as separate fields. The integral operators are defined by taking averages
\[
\left[\begin{array}{cc} \mathrm V & \mathrm K   \end{array}\right]:=\ave{\gamma\punto} \left[\begin{array}{cc}  \mathrm S_u & \mathrm D_u  \end{array}\right], \qquad 
\left[\begin{array}{cc} \mathrm K^t & \mathrm W   \end{array}\right]:=\ave{\mathbf t(\punto)} \left[\begin{array}{cc}  \mathrm S & -\mathrm D  \end{array}\right].
\]

\begin{proposition}[Mapping properties]
The following operators are bounded:
\begin{eqnarray*}
(\mathrm S_u,\mathrm S_p) &:& \mathbf H^{-1/2}_0(\Gamma) \to \mathbf W(\mathbb R^2)\times L^2(\mathbb R^2),\\
(\mathrm D_u,\mathrm D_p) &:& \mathbf H^{1/2}(\Gamma) \to \mathbf W(\mathbb R^2\setminus\Gamma)\times L^2(\mathbb R^2),\\
\mathrm V &:& \mathbf H^{-1/2}_0(\Gamma) \to \mathbf H^{1/2}(\Gamma),\\
\mathrm K^t &:& \mathbf H^{-1/2}_0(\Gamma) \to \mathbf H^{-1/2}(\Gamma),\\
\mathrm K &:& \mathbf H^{1/2}(\Gamma) \to \mathbf H^{1/2}(\Gamma),\\
\mathrm W &:& \mathbf H^{1/2}(\Gamma) \to \mathbf H^{-1/2}(\Gamma).
\end{eqnarray*}
\end{proposition}

\begin{proposition}[Jump relations]
For all $\boldsymbol\lambda\in \mathbf H^{-1/2}_0(\Gamma)$ and $\boldsymbol\varphi \in \mathbf H^{1/2}(\Gamma)$,
\begin{eqnarray*}
\jump{\gamma \mathrm S_u\boldsymbol\lambda}=0, & & \jump{\gamma D_u\boldsymbol\varphi}=-\boldsymbol\varphi,\\
\jump{\mathbf t(\mathrm S_u\boldsymbol\lambda,\mathrm S_p\boldsymbol\lambda)}=\boldsymbol\lambda, & &
\jump{\mathbf t(\mathrm D_u\boldsymbol\varphi,\mathrm D_p\boldsymbol\varphi)}=0,\\
\mathbf t^\pm(\mathrm S_u\boldsymbol\lambda,\mathrm S_p\boldsymbol\lambda)=\mp\smallfrac12\boldsymbol\lambda+\mathrm K^t\boldsymbol\lambda, & & 
\mathbf t^\pm (\mathrm D_u\boldsymbol\varphi,\mathrm D_p\boldsymbol\varphi)=\pm\smallfrac12 \boldsymbol\varphi + \mathrm K\boldsymbol\varphi.
\end{eqnarray*}
\end{proposition}

\paragraph*{Simple solutions.}
With $\boldsymbol\lambda=\mathbf n$ and $\boldsymbol\varphi=0$, we get
\begin{equation}\label{eq:9.44}
\mathrm S_u \mathbf n= \mathbf 0, \quad \mathrm S_p\mathbf n=-\chi_{\Omega_-}, \quad \mathrm V\mathbf n=\mathbf 0, \quad -\smallfrac12\mathbf n+\mathrm K^t\mathbf n=\mathbf 0.
\end{equation}
With $\boldsymbol\lambda=\mathbf 0$ and $\boldsymbol\varphi=\mathbf c\in \mathcal M_\Gamma=\gamma \mathcal M$, we get
\begin{equation}\label{eq:9.5}
\mathrm D_u\mathbf c=-\chi_{\Omega_-}\mathbf c, \quad \mathrm D_p\mathbf c=0, \quad \mathrm W\mathbf c=\mathbf 0, \quad \smallfrac12\mathbf c+\mathrm K\mathbf c=\mathbf 0.
\end{equation}

\begin{proposition}[Properties of $\mathrm V$]
For all $\boldsymbol\lambda,\boldsymbol\mu\in \mathbf H^{-1/2}_0(\Gamma)$,
\begin{equation}\label{eq:9.6}
\langle \boldsymbol\mu,\mathrm V\boldsymbol\lambda\rangle_\Gamma = \langle \boldsymbol\lambda,\mathrm V\boldsymbol\mu\rangle_\Gamma, \qquad \langle \boldsymbol\lambda,\mathrm V\boldsymbol\lambda\rangle_\Gamma\ge 0
\end{equation}
and there exists $C_\Gamma>0$ such that
\[
\langle \boldsymbol\lambda,\mathrm V\boldsymbol\lambda\rangle_\Gamma\ge C_\Gamma \|\boldsymbol\lambda\|_{-1/2,\Gamma}^2 \qquad \forall \boldsymbol\lambda\in \mathbf H^{-1/2}_m(\Gamma)\cap \mathbf H^{-1/2}_0(\Gamma).
\]
Also
\[
\mathrm{Ker}\,\mathrm V =\mathrm{span}\,\{ \mathbf n\} \qquad \mbox{and} \qquad \mathrm{Range}\, \mathrm V \subset  \mathbf H^{1/2}_n(\Gamma).
\]
\end{proposition}

\begin{proof} The proof of Propositions \ref{prop:5.5} and \ref{prop:5.7} (same results in the three dimensional case) are valid.
\end{proof}

\paragraph{The range of $\mathrm V$.} Identifying the range of $\mathrm V$ requires some additional work, given the fact that we have been defining $\mathrm V$ on $\mathbf H^{-1/2}_0(\Gamma)$, as opposed to having it defined on the entire $\mathbf H^{-1/2}(\Gamma)$. To do that we consider the solution of the coercive problems ($j=1,2$)
\[
\left[\begin{array}{cc} \boldsymbol\phi_{\mathrm{eq}}^j\in \mathbf H^{-1/2}(\Gamma), \qquad\langle \boldsymbol\phi_{\mathrm{eq}}^j,\mathbf e_j\rangle_\Gamma=1,\\[1.3ex]
\langle \boldsymbol\phi_{\mathrm{eq}}^j,\mathrm V\boldsymbol\mu\rangle_\Gamma = 0 \qquad \forall \boldsymbol\mu\in \mathbf H^{-1/2}(\Gamma),
\end{array}\right.
\]
where $\{\mathbf e_1,\mathbf e_2\}$ is the canonical basis of $\mathbb R^2$. These densities are a basis of the space of equilibrium distributions for the Stokes problem (see \cite{DoSa:2006} and \cite{DiKaMa:2011}). They are the Stokes equivalent of the electrostatic equilibrium distribution that leads to the definition of logarithmic capacity. Using a coercivity argument (see the proof of Proposition \ref{prop:5.7}), it is easy to see that 
\[
\mathrm V:\mathbf H^{-1/2}_0(\Gamma)\cap \mathbf H^{-1/2}_m(\Gamma) \longrightarrow \mathrm{span}\,\{\mathbf n,\boldsymbol\phi_{\mathrm{eq}}^1,\boldsymbol\phi_{\mathrm{eq}}^2\}^\circ
\]
is invertible.

\begin{proposition}[Properties of $\mathrm W$] 
For all $\boldsymbol\varphi,\boldsymbol\psi\in \mathbf H^{1/2}(\Gamma)$,
\begin{equation}\label{eq:10.3}
\langle \mathrm W\boldsymbol\varphi,\boldsymbol\psi\rangle_\Gamma = \langle \mathrm W\boldsymbol\psi,\boldsymbol\varphi\rangle_\Gamma, \qquad \langle \mathrm W\boldsymbol\varphi, \boldsymbol\varphi\rangle_\Gamma\ge 0
\end{equation}
and there exists $C_\Gamma$ such that
\[
\langle \mathrm W\boldsymbol\varphi,\boldsymbol\varphi\rangle_\Gamma \ge C_\Gamma\|\boldsymbol\varphi\|_{1/2,\Gamma}^2 \qquad \forall \boldsymbol\varphi\in \mathbf H^{1/2}_{\mathcal M}(\Gamma):=\{\boldsymbol\xi\,:\,\int_\Gamma \boldsymbol\xi\cdot\mathbf c=0\quad\forall \mathbf c\in \mathcal M_\Gamma\}.
\]
Also
\[
\mathrm{Ker}\,\mathrm W =\mathcal M_\Gamma \qquad \mbox{and} \qquad \mathrm{Range}\, \mathrm W  = \mathcal M_\Gamma^\circ.
\]
\end{proposition}

\begin{proof} We follow the proof of Proposition \ref{prop:6.5}. 
Let $(\mathbf u_\varphi,p_\varphi):=(\mathrm D_u\boldsymbol\varphi,\mathrm D_p\boldsymbol\varphi)$ and $\mathbf u_\psi:=\mathrm D_u\boldsymbol\psi$. Then we can prove that
\begin{equation}\label{eq:9.7}
\langle\mathrm W\boldsymbol\varphi,\boldsymbol\psi\rangle_\Gamma = a_{\mathbb R^2\setminus\Gamma}(\mathbf u_\varphi,\mathbf u_\psi),
\end{equation}
which implies symmetry and positive semi-definiteness \eqref{eq:10.3}. 

Formula \eqref{eq:9.5} shows that $\mathcal M_\Gamma\subset \mathrm{Ker}\,\mathrm W$.
If $\mathrm W\boldsymbol\varphi=\mathbf 0$, then, by \eqref{eq:9.7}, it follows that $\boldsymbol\varepsilon(\mathbf u_\varphi)=0$ in $\mathbb R^2\setminus\Gamma$. By Lemmas \ref{lemma:rigid} and \ref{prop:8.4b}, it follows that $\mathbf u_\varphi|_{\Omega_-}\in \mathcal M$ and $\mathbf u_\varphi|_{\Omega_+}\in \boldsymbol P_0(\Omega_+)$. Therefore $\boldsymbol\varphi=-\jump{\gamma\mathbf u_\varphi} \in \mathcal M_\Gamma$. 
The remainder of the proof of Proposition \ref{prop:6.5} can be applied verbatim.

To prove coercivity, we proceed as in the proof of Proposition \ref{prop:6.6}. We first need to prove that
\[
|\!|\!|\mathbf u|\!|\!|^2:=\|\boldsymbol\varepsilon(\mathbf u)\|_{\mathbb R^2\setminus\Gamma}^2+|\boldsymbol\jmath(\mathbf u)+\boldsymbol\ell_{\mathrm D}(\jump{\gamma\mathbf u})|^2+\sum_{\ell=1}^3 \Big| \int_\Gamma \jump{\gamma\mathbf u}\cdot\mathbf c_\ell\Big|^2,\quad \mbox{where } \mathcal M=\mathrm{span}\{\mathbf c_1,\mathbf c_2,\mathbf c_3\},
\]
is a norm in $\mathbf W(\mathbb R^2\setminus\Gamma)$. This can be done using Lemmas \ref{lemma:rigid} and \ref{prop:8.4b}.
Using Korn's inequality and a compactness argument, we can show that this is an equivalent norm in $\mathbf W(\mathbb R^2\setminus\Gamma)$. However, $|\!|\!|\mathbf u|\!|\!|=\|\boldsymbol\varepsilon(\mathbf u_\varphi)\|_{\mathbb R^2\setminus\Gamma}$ if $\boldsymbol\varphi\in \mathbf H^{1/2}_{\mathcal M}(\Gamma)$. The remainder of the proof of Proposition \ref{prop:6.6} can be then applied to show coercivity.

Finally, using coercivity and the fact that $\mathrm{Range}\,\mathrm W \subset \mathcal M_\Gamma^\circ$, it is easy to see that both sets are equal.
\end{proof}

\begin{proposition}[$\mathrm K^t$ is the transpose of $\mathrm K$]\label{prop:10.11} 
\[
\langle \boldsymbol\lambda,\mathrm K\boldsymbol\varphi\rangle_\Gamma= \langle \mathrm K^t\boldsymbol\lambda,\boldsymbol\varphi\rangle_\Gamma \qquad \forall \boldsymbol\lambda\in \mathbf H^{-1/2}(\Gamma), \quad \boldsymbol\varphi\in \mathbf H^{1/2}(\Gamma).
\]
\end{proposition}

\begin{proof} 
We first prove that single and double layer potentials are orthogonal with respect to the semi-inner product $a_{\mathbb R^2\setminus\Gamma}(\punto,\punto)$, that is,
\[
\mathbf u=\mathrm S_u\boldsymbol\lambda, \quad \mathbf v=\mathrm D_u\boldsymbol\varphi \quad \Longrightarrow\quad
a_{\mathbb R^2\setminus\Gamma}(\mathbf u,\mathbf v)=0.
\]
This is done by reversing the roles of $\mathbf u$ and $\mathbf v$ in \eqref{eq:9.3}, given the fact that $\mathrm S_u\boldsymbol\lambda\in \mathbf W(\mathbb R^2)$ can be used as a test function.
The remainder of the argument follows the proof of Proposition \ref{prop:6.8}.
\end{proof}

\begin{proposition}[Representation formula]\label{prop:11.2}
Let $\mathbf u\in \mathbf W(\mathbb R^2\setminus\Gamma)$ and $p\in L^2(\mathbb R^2)$ satisfy
\begin{alignat*}{4}
 -2\nu \mathrm{div}\,\boldsymbol\varepsilon(\mathbf u)+\nabla p = \mathbf 0 & \qquad & \mbox{in $\mathbb R^2\setminus\Gamma$},\\
 \mathrm{div}\,\mathbf u = 0 & & \mbox{in $\mathbb R^2\setminus\Gamma$}.
\end{alignat*}
Then there exists $\mathbf u_\infty\in \boldsymbol P_0(\mathbb R^2)$ such that
\begin{equation}
\mathbf u = \mathbf u_\infty+ \mathrm S_u \jump{\mathbf t(\mathbf u,p)}-\mathrm D_u \jump{\gamma\mathbf u}\quad\mbox{and}\quad
p = \mathrm S_p \jump{\mathbf t(\mathbf u,p)}-\mathrm D_p \jump{\gamma\mathbf u}.
\end{equation}
\end{proposition}

\begin{proof} See the proof of Proposition \ref{prop:7.2}. The argument uses the fact that the unique solutions to \eqref{eq:9.1} with vanishing right hand side are the elements of $\boldsymbol P_0(\mathbb R^2)\times \{0\}$ (Proposition \ref{prop:AA}).
\end{proof}

\section{Review \#2: Laplace and Lam\'e potentials}

In this section we give a fast review of some known properties of the layer potentials for the Laplace and Lam\'e (or Navier, or homogeneous isotropic linear elasticity) equations. The variational approach (with weighted Sobolev spaces, as in Sections \ref{sec:5.1}, \ref{sec:6.1}, and \ref{sec:2d}) can be found in \cite{Nedelec:1971}, although the fine print of relating integral forms (what we will do in the remainder of the paper) is missing. A purely integral theory is developed in \cite{Costabel:1988} and \cite[Chapter 8 \& Chapter 10]{McLean:2000}. Note that the variational theory gives fast and elegant proofs of mapping properties in the basic variational setting but needs some additional work in order to show that the variational definitions correspond to weak integral potentials.

\paragraph{Laplace potentials.} For questions of easiness of reference, we will make copies of the layer potentials for the Laplacian so that they act componentwise on vector-valued densities. The single layer and double layer potentials for the Laplacian are defined for $\mathbf x\in \mathbb R^d\setminus\Gamma$ by 
\[
(\mathrm S^\Delta \boldsymbol\lambda)(\mathbf x):=\langle\mathrm E^\Delta (\mathbf x-\cdot),\boldsymbol\lambda\rangle_\Gamma,\qquad
(\mathrm D^\Delta \boldsymbol\varphi)(\mathbf x):=\int_\Gamma \mathrm T^\Delta(\mathbf x-\mathbf y;\mathbf n(\mathbf y))\,\boldsymbol\varphi(\mathbf y)\mathrm d\Gamma(\mathbf y),
\]
where
\begin{alignat*}{4}
\mathrm E^\Delta(\mathbf r)& :=\frac1{2\pi(d-1)}\psi (r)\mathrm I, & \qquad &\psi(r):=\left\{ \begin{array}{ll} \log r^{-1}, & d=2,\\ 1/r, & d=3,\end{array}\right.
\\
\mathrm T^\Delta(\mathbf r;\mathbf n)&:=\frac1{2\pi(d-1)}\frac1{r^d}(\mathbf r\cdot\mathbf n)\mathrm I,
\end{alignat*}
contain diagonal copies of the Laplacian monopole and dipole distributions, and $r=|\mathbf r|$. For some mapping properties, we will use the spaces $\mathbf H^1_{\mathrm{loc}}(\mathbb R^2)$. These spaces are not normed spaces but are easily seen to be metrizable, by using a sequence of cut-off functions.

\begin{proposition}\label{prop:8.1}
For all $(\boldsymbol\lambda,\boldsymbol\varphi)\in \mathbf H^{-1/2}(\Gamma)\times \mathbf H^{1/2}(\Gamma)$, the function $\mathbf u:=\mathrm S^\Delta\boldsymbol\lambda-\mathrm D^\Delta\boldsymbol\varphi$ is a solution of
\begin{subequations}
\begin{alignat}{4}
 -\Delta \mathbf u =  \mathbf 0 & \qquad & \mbox{in $\mathbb R^d\setminus\Gamma$},\\
 \jump{\gamma  \mathbf u}=\boldsymbol\varphi, & & \\
 \jump{(\mathrm D\mathbf  u)\mathbf n}= \boldsymbol\lambda. & &
\end{alignat}
\end{subequations}
Moreover, the following mappings are continuous:
\begin{equation}
\mathrm D^\Delta:\mathbf H^{1/2}(\Gamma) \to \mathbf W(\mathbb R^d\setminus\Gamma), \qquad \mathrm S^\Delta:\left\{\begin{array}{ll} \mathbf H^{-1/2}(\Gamma) \to \mathbf W(\mathbb R^3), & d=3,\\
\mathbf H^{-1/2}_0(\Gamma) \to \mathbf W(\mathbb R^2), &d=2,\\
\mathbf H^{-1/2}(\Gamma) \to \mathbf H^1_{\mathrm{loc}}(\mathbb R^2), &d=2.
\end{array}\right.
\end{equation}
\end{proposition}

\paragraph{Lam\'e potentials.} For reasons that will become apparent at the time of comparing potentials with the Stokes problem, all Lam\'e potentials will be written in terms of one of the Lam\'e parameters $\mu$ and the non-physical quantity:
\[
A:=\frac{\mu}{\lambda+2\mu}
\]
The layer potentials are defined for $\mathbf x\in \mathbb R^d\setminus\Gamma$ by
\[
(\mathrm S^{\mathrm L} \boldsymbol\lambda)(\mathbf x) :=\langle\mathrm E^{\mathrm L} (\mathbf x-\cdot),\boldsymbol\lambda\rangle_\Gamma,\qquad
(\mathrm D^{\mathrm L} \boldsymbol\varphi)(\mathbf x):=\int_\Gamma \mathrm T^{\mathrm L}(\mathbf x-\mathbf y;\mathbf n(\mathbf y))\,\boldsymbol\varphi(\mathbf y)\mathrm d\Gamma(\mathbf y),
\]
where (see \cite[(2.2.21)]{HsWe:2008} for the definition of $\mathrm T^{\mathrm L}$)
\begin{alignat*}{4}
\mathrm E^{\mathrm L}(\mathbf r)& :=\frac1{4\pi(d-1)\mu }\Big( (1+A) \psi (r)\mathrm I+(1-A)\frac1{r^d}\mathbf r\otimes \mathbf r\Big),
\\
\mathrm T^{\mathrm L}(\mathbf r;\mathbf n)&:=\frac1{2\pi(d-1)}\left( \frac{A}{r^d} \big( (\mathbf r\cdot\mathbf n)\mathrm I -\mathbf r\otimes\mathbf n+\mathbf n\otimes \mathbf r\Big) +(1-A) \frac{d}{r^{d+2}} (\mathbf r\cdot\mathbf n)\mathbf r\otimes \mathbf r\right).
\end{alignat*}
Note that the definition of the Lam\'e dipoles $\mathrm T^{\mathrm L}(\punto;\mathbf n)$ can be found with the process
\begin{eqnarray*}
\mathbb R^d\ni \mathbf d,\mathbf n   \longmapsto  \mathrm T^{\mathrm L}(\punto;\mathbf n)^\top\mathbf d & =& -\boldsymbol\sigma^{\mathrm L}(\mathrm E^{\mathrm L}\mathbf d)\mathbf n= -\big( 2\mu \boldsymbol\varepsilon(\mathrm E^{\mathrm L}\mathbf d)\mathbf n+\lambda\,\mathrm{div}\,(\mathrm E^{\mathrm L}\mathbf d)\,\mathbf n\Big)\\
&= & -\frac{\mu}{A} \Big( 2\,A \, \boldsymbol\varepsilon(\mathrm E^{\mathrm L}\mathbf d)\mathbf n + (1-2\,A) \mathrm{div}\,(\mathrm E^{\mathrm L}\mathbf d)\,\mathbf n\Big),
\end{eqnarray*}
where $\mathbf d$ acts as a dipole direction and $\mathbf n$ as the stress direction.

\begin{proposition}\label{prop:8.2}
For all $(\boldsymbol\lambda,\boldsymbol\varphi)\in \mathbf H^{-1/2}(\Gamma)\times \mathbf H^{1/2}(\Gamma)$, the function $\mathbf u:=\mathrm S^{\mathrm L}\boldsymbol\lambda-\mathrm D^{\mathrm L}\boldsymbol\varphi$ is a solution of
\begin{subequations}
\begin{alignat}{4}
 -\mathrm{div} \,\boldsymbol\sigma^{\mathrm L}(\mathbf u) =  \mathbf 0 & \qquad & \mbox{in $\mathbb R^d\setminus\Gamma$},\\
 \jump{\gamma  \mathbf u}=\boldsymbol\varphi, & & \\
 \jump{\boldsymbol\sigma^{\mathrm L}(\mathbf  u)\mathbf n}= \boldsymbol\lambda. & &
\end{alignat}
\end{subequations}
Moreover, the following mappings are continuous:
\begin{equation}
\mathrm D^{\mathrm L}:\mathbf H^{1/2}(\Gamma) \to \mathbf W(\mathbb R^d\setminus\Gamma), \qquad \mathrm S^{\mathrm L}:\left\{\begin{array}{ll} \mathbf H^{-1/2}(\Gamma) \to \mathbf W(\mathbb R^3), & d=3,\\
\mathbf H^{-1/2}_0(\Gamma) \to \mathbf W(\mathbb R^2), &d=2,\\
\mathbf H^{-1/2}(\Gamma) \to \mathbf H^1_{\mathrm{loc}}(\mathbb R^2), &d=2.
\end{array}\right.
\end{equation}
\end{proposition}

Note how, if we write $\boldsymbol\sigma^\Delta(\mathbf u):=\mathrm D \mathbf u$, Propositions \ref{prop:8.1} and \ref{prop:8.2} are almost identical twins.

\section{Stokes from Laplace and Lam\'e}

To close this exposition of the Stokes layer potentials, we show that their integral expressions give the same functions as those defined in Sections \ref{sec:5.1}, \ref{sec:6.1} and \ref{sec:2d}. We will avoid having a redundant name for two entities that end up being the same, by ignoring that we ever defined the variational potentials in the previous sections. Consider first the pressure part of the potentials (as usual, $\mathbf x \in \mathbb R^d\setminus\Gamma$)
\begin{equation}\label{eq:11.20}
(\mathrm S_p \boldsymbol\lambda)(\mathbf x):=\langle \boldsymbol\lambda,\mathbf e_p(\mathbf x-\cdot)\rangle_\Gamma, \qquad 
(\mathrm D_p \boldsymbol\varphi)(\mathbf x):=\int_\Gamma \mathbf t_p(\mathbf x-\mathbf y;\mathbf n(\mathbf y))\cdot\boldsymbol\varphi(\mathbf y)\mathrm d\Gamma(\mathbf y),
\end{equation}
where
\begin{eqnarray*}
\mathbf e_p(\mathbf r)&:=& \frac1{2\pi (d-1)\,r^d}\mathbf r=-\frac1{2\pi (d-1)} \nabla \psi(r),\\
\mathbf t_p(\mathbf r;\mathbf n)&:=& \frac1{2\pi (d-1)}{2\nu}\left(\frac{d}{r^{d+2}}\mathbf r\otimes \mathbf r-\frac1{r^2}\mathrm I\right)\mathbf n.
\end{eqnarray*}
It is easy to see that both potentials define functions with $\mathcal C^\infty(\mathbb R^d\setminus\Gamma)$ components.

\begin{proposition}\label{prop:11.1}
The following identities hold in $\mathbb R^d\setminus\Gamma$ for arbitrary $\boldsymbol\lambda\in \mathbf H^{-1/2}(\Gamma)$ and $\boldsymbol\varphi\in \mathbf H^{1/2}(\Gamma)$:
\begin{equation}\label{eq:11.1}
\mathrm S_p \boldsymbol\lambda=-\mathrm{div}\,\mathrm S^\Delta \boldsymbol\lambda, \qquad \mathrm D_p\boldsymbol\varphi=-(2\nu) \mathrm{div}\,\mathrm D^\Delta \boldsymbol\varphi.
\end{equation}
Therefore, the following maps are continuous:
\begin{equation}\label{eq:11.2}
\mathrm D_p:\mathbf H^{1/2}(\Gamma) \to   L^2(\mathbb R^d), \qquad \mathrm S_p:\left\{\begin{array}{ll} \mathbf H^{-1/2}(\Gamma) \to   L^2(\mathbb R^3), & d=3,\\
\mathbf H^{-1/2}_0(\Gamma) \to   L^2(\mathbb R^2), &d=2,\\
\mathbf H^{-1/2}(\Gamma) \to   L^2_{\mathrm{loc}}(\mathbb R^2), &d=2.
\end{array}\right.
\end{equation}
\end{proposition}

\begin{proof}
The first part follows by a simple computation. Mapping properties are a simple consequence of Proposition \ref{prop:8.1} and of \eqref{eq:11.1}.
\end{proof}

The velocity part of the Stokes potentials is defined by respective superposition of stokeslets and stresslets:
\begin{equation}\label{eq:11.21}
(\mathrm S_u \boldsymbol\lambda)(\mathbf x) :=\langle\mathrm E_u (\mathbf x-\cdot),\boldsymbol\lambda\rangle_\Gamma, \qquad 
(\mathrm D_u \boldsymbol\varphi)(\mathbf x):=\int_\Gamma \mathrm T_u(\mathbf x-\mathbf y;\mathbf n(\mathbf y))\,\boldsymbol\varphi(\mathbf y)\mathrm d\Gamma(\mathbf y),
\end{equation}
where
\begin{subequations}\label{eq:11.31}
\begin{alignat}{4}
\mathrm E_u(\mathbf r)& :=\frac1{4\pi(d-1)\nu }\Big(  \psi (r)\mathrm I+\frac1{r^d}\mathbf r\otimes \mathbf r\Big),
\\
\mathrm T_u(\mathbf r;\mathbf n)&:=\frac1{2\pi(d-1)}\frac{d}{r^{d+2}} \,(\mathbf r\cdot\mathbf n)\,\mathbf r\otimes \mathbf r.
\end{alignat}
\end{subequations}
For some forthcoming arguments, it will be useful to have the rotlet distributions at hand
\begin{equation}\label{eq:11.22}
(\mathrm R\boldsymbol\varphi)(\mathbf x):=\int_\Gamma \mathrm M(\mathbf x-\mathbf y;\mathbf n(\mathbf y))\,\boldsymbol\lambda(\mathbf y)\mathrm d\Gamma(\mathbf y),
\end{equation}
where
\[
\mathrm M(\mathbf r;\mathbf n):=\frac1{2\pi (d-1) r^d}\,(\mathbf r\otimes\mathbf n-\mathbf n\otimes\mathbf r).
\]
A simple computation shows that in $\mathbb R^d\setminus\Gamma$:
\begin{equation}\label{eq:11.3}
\mathrm{div}\,\mathrm R\boldsymbol\varphi=-\mathrm{div}\,\mathrm D^\Delta\boldsymbol\varphi=(2\nu)^{-1} \mathrm D_p\boldsymbol\varphi \qquad \forall \boldsymbol\varphi\in \mathbf H^{1/2}(\Gamma).
\end{equation}

\begin{proposition}\label{prop:11.2}
The following mappings are continuous:
\begin{equation}
\mathrm D_u:\mathbf H^{1/2}(\Gamma) \to \mathbf W(\mathbb R^d\setminus\Gamma), \qquad \mathrm S_u:\left\{\begin{array}{ll} \mathbf H^{-1/2}(\Gamma) \to \mathbf W(\mathbb R^3), & d=3,\\
\mathbf H^{-1/2}_0(\Gamma) \to \mathbf W(\mathbb R^2), &d=2,\\
\mathbf H^{-1/2}(\Gamma) \to \mathbf H^1_{\mathrm{loc}}(\mathbb R^2), &d=2.
\end{array}\right.
\end{equation}
Moreover
\begin{equation}\label{eq:11.5}
\jump{\gamma\mathrm D_u\boldsymbol\varphi}=-\boldsymbol\varphi \qquad \forall\boldsymbol\varphi \in \mathbf H^{1/2}(\Gamma).
\end{equation}
\end{proposition}

\begin{proof}
Note first that
\begin{equation}\label{eq:11.6}
\mathrm S^{\mathrm L}=\mu^{-1}A \mathrm S^\Delta +(1-A)  \nu\,\mu^{-1}\,\mathrm S_u,
\end{equation}
which transmits the mapping properties of the Laplace and Lam\'e single layer potentials (Propositions \ref{prop:8.1} and \ref{prop:8.2}) to $\mathrm S_u$. Note that this includes the fact that $\jump{\gamma \mathrm S_u\boldsymbol\lambda}=\mathbf 0$. We next notice that
\begin{equation}\label{eq:11.7}
\mathrm D^{\mathrm L}=A\,\mathrm D^\Delta - A\, \mathrm R+(1-A)\mathrm D_u=\mathrm D_u+A\,(\mathrm D^\Delta-\mathrm R-\mathrm D_u).
\end{equation}
Since $\mathrm D^{\mathrm L}$ is bounded $\mathbf H^{1/2}(\Gamma)\to \mathbf W(\mathbb R^d\setminus\Gamma)$ for all mathematically valid choices of the Lam\'e parameters (and then at least for all $0<A<1$), it is clear that $\mathrm D_u$ and $\mathrm D^\Delta-\mathrm R-\mathrm D_u$ have the same mapping property. Note that this implies that the rotlet operator is bounded between the same spaces as well. Taking the jump of the trace on both sides of \eqref{eq:11.7} and using Propositions \ref{prop:8.1} and \ref{prop:8.2} it follows that
\[
-\boldsymbol\varphi = - A\boldsymbol\varphi-A \jump{\gamma\mathrm R\boldsymbol\varphi}+(1-A) \jump{\gamma\mathrm D_u\boldsymbol\varphi},
\]
again for all $0<A<1$, and therefore \eqref{eq:11.5} holds. Note that, in passing, we have proved that $\jump{\gamma\mathrm R\boldsymbol\varphi}=\mathbf 0$, and therefore $\mathrm R:\mathbf H^{1/2}(\Gamma)\to \mathbf W(\mathbb R^d)$ is bounded.
\end{proof}

\begin{proposition}\label{prop:11.3}
Let $(\boldsymbol\lambda,\boldsymbol\varphi) \in \mathbf H^{-1/2}(\Gamma)\times \mathbf H^{1/2}(\Gamma)$ and let
\[
\mathbf u:=\mathrm S_u\boldsymbol\lambda-\mathrm D_u\boldsymbol\varphi, \qquad p:=\mathrm S_p\boldsymbol\lambda-\mathrm D_p\boldsymbol\varphi.
\]
Then
\begin{subequations}\label{eq:11.8}
\begin{alignat}{6}
\label{eq:11.8a}
(\mathbf u,p)\in \mathbf H^1_{\mathrm{loc}}(\mathbb R^d\setminus\Gamma) \times L^2_{\mathrm{loc}}(\mathbb R^d),\\
\label{eq:11.8b}
 -2\nu \mathrm{div}\,\boldsymbol\varepsilon(\mathbf u)+\nabla p = \mathbf 0 & \qquad & \mbox{in $\mathbb R^d\setminus\Gamma$},\\
\label{eq:11.8c}
 \mathrm{div}\,\mathbf u = 0 & & \mbox{in $\mathbb R^d\setminus\Gamma$},\\
\label{eq:11.8d}
 \jump{\gamma\mathbf u}=\boldsymbol\varphi, & & \\
\label{eq:11.8e}
 \jump{\mathbf t(\mathbf u,p)}= \boldsymbol\lambda. & &
\end{alignat}
Moreover, $(\mathbf u,p)\in \mathbf W(\mathbb R^d\setminus\Gamma)\times L(\mathbb R^d)$ if $d=3$, or if $\boldsymbol\lambda \in \mathbf H^{-1/2}_0(\Gamma)$ and $d=2$. Finally, when $d=2$, condition \eqref{eq:9.52} is satisfied.
\end{subequations}
\end{proposition}

\begin{proof}
 Propositions \ref{prop:11.1} and \ref{prop:11.2} prove \eqref{eq:11.8a} and \eqref{eq:11.8d}. The differential equations \eqref{eq:11.8b} and \eqref{eq:11.8d} can be proved to hold by a direct (but cumbersome) computation. 

Let now $\boldsymbol\lambda \in \mathbf H^{-1/2}(\Gamma)$, take $\mu=\nu$ and $\lambda=0$ in the Lam\'e equations (therefore $A=\frac12$),  define
\[
\mathbf u:=\mathrm S_u\boldsymbol\lambda, \qquad p:=\mathrm S_p\boldsymbol\lambda, \qquad \mathbf u^{\mathrm L}:=\mathrm S^{\mathrm L}\boldsymbol\lambda, \qquad \mathbf u^\Delta :=\mathrm S^\Delta \boldsymbol\lambda,
\]
and note that by \eqref{eq:11.6} and Proposition \ref{prop:11.1}
\[
\mathbf u=2\mathbf u^{\mathrm L}-\nu^{-1}\mathbf u^\Delta \quad \mbox{and}\quad p=-\mathrm{div}\,\mathbf u^\Delta.
\]
By Propositions \ref{prop:8.1} and \ref{prop:8.2} and the definitions of the normal stresses, it follows that
\[
2\nu (\boldsymbol\varepsilon(\mathbf u^{\mathrm L}),\mathrm D \mathbf v)_{\mathbb R^d} =\langle\boldsymbol\lambda,\gamma \mathbf v\rangle_\Gamma, \qquad (\mathrm D\mathbf u^\Delta,\mathrm D\mathbf v)_{\mathbb R^d}=\langle\boldsymbol\lambda,\gamma\mathbf v\rangle_\Gamma, \qquad \forall \mathbf v\in \mathbf H^1_{\mathrm{comp}}(\mathbb R^d),
\]
where the subscript 'comp' is used to denote compact support. Therefore
\begin{eqnarray*}
\langle\jump{\mathbf t(\mathbf u,p)},\gamma\mathbf v\rangle_\Gamma &=& 2\nu (\boldsymbol\varepsilon(\mathbf u),\mathrm D \mathbf v)_{\mathbb R^d}-(p,\mathrm{div}\,\mathbf v)_{\mathbb R^d}\\
&=& 4\nu (\boldsymbol\varepsilon(\mathbf u^{\mathrm L}),\mathrm D\mathbf v)_{\mathbb R^d}-2(\boldsymbol\varepsilon(\mathbf u^\Delta),\mathrm D\mathbf v)_{\mathbb R^d}+(\mathrm{div}\,\mathbf u^\Delta,\mathrm{div}\,\mathbf v)_{\mathbb R^d}\\
&=& \langle \boldsymbol\lambda,\gamma\mathbf v\rangle_\Gamma -(\mathrm D\mathbf u^\Delta,(\mathbf D\mathbf v)^\top)_{\mathbb R^d}+(\mathrm{div}\,\mathbf u^\Delta,\mathrm{div}\,\mathbf v)_{\mathbb R^d},
\end{eqnarray*} 
for all $\mathbf v\in \mathbf H^1_{\mathrm{comp}}(\mathbb R^d).$ However, it can easily be seen with a density argument and differentiation in the sense of distributions that
\[
(\mathrm D\mathbf u^\Delta,(\mathbf D\mathbf v)^\top)_{\mathbb R^d}=(\mathrm{div}\,\mathbf u^\Delta,\mathrm{div}\,\mathbf v)_{\mathbb R^d}
\qquad \forall \mathbf u\in \mathbf H^1_{\mathrm{loc}}(\mathbb R^d), \, \mathbf v\in \mathbf H^1_{\mathrm{comp}}(\mathbb R^d).
\]
This proves that
\[
\jump{\mathbf t(\mathrm S_u\boldsymbol\lambda, \mathrm S_p\boldsymbol\lambda)}=\boldsymbol\lambda \qquad \forall \boldsymbol\lambda \in \mathbf H^{-1/2}(\Gamma).
\]

We start afresh with $\boldsymbol\varphi\in \mathbf H^{1/2}(\Gamma)$, take general $A$ and $\mu$, and define
\[
\mathbf u:=\mathrm D_u\boldsymbol\varphi, \quad p:=\mathrm D_p\boldsymbol\varphi, \quad \mathbf u^{\mathrm L}:=\mathrm D^{\mathrm L}\boldsymbol\varphi, \quad \mathbf u^\Delta:=\mathrm D^\Delta \boldsymbol\varphi, \quad \mathbf u^{\mathrm R}:=\mathrm R\boldsymbol\varphi,
\]
and note that by \eqref{eq:11.3} and \eqref{eq:11.7}, it follows that
\begin{equation}\label{eq:11.9}
\mathrm{div}\,\mathbf u^{\mathrm R}-\mathrm{div}\,\mathbf u^\Delta=\nu^{-1} \,p \quad \mbox{and}\quad \mathbf u^{\mathrm L}=A\,\mathbf u^\Delta-A\,\mathbf u^{\mathrm R}+(1-A)\mathbf u.
\end{equation}
Noticing that
\[
A\,\mu^{-1} \boldsymbol\sigma^{\mathrm L}(\mathbf v)=2A \boldsymbol\varepsilon(\mathbf v)+(1-2A)(\mathrm{div}\,\mathbf v)\mathrm I, 
\]
using \eqref{eq:11.9} and the fact that $\mathrm{div}\,\mathbf u=0$, we can expand
\begin{eqnarray*}
A\mu^{-1} \boldsymbol\sigma^{\mathrm L}(\mathbf u^{\mathrm L}) &=& 2 A^2\Big(\boldsymbol\varepsilon(\mathbf u^\Delta)-\boldsymbol\varepsilon(\mathbf u^{\mathrm R})-\boldsymbol\varepsilon(\mathbf u)-(\mathrm{div}\mathbf u^\Delta)\mathrm I+(\mathrm{div}\,\mathbf u^{\mathrm R})\mathrm I+(\mathrm{div}\,\mathbf u)\mathrm I\Big)\\
& &+ A\,\Big( 2\boldsymbol\varepsilon(\mathbf u)+(\mathrm{div}\,\mathbf u^\Delta)\mathrm I-(\mathrm{div}\,\mathbf u^{\mathrm R})\mathrm I-(3\, \mathrm{div}\,\mathbf u)\mathrm I\Big)+(\mathrm{div}\,\mathbf u)\mathrm I\\
&=& 2A^2 \Big(\boldsymbol\varepsilon(\mathbf u^\Delta)-\boldsymbol\varepsilon(\mathbf u^{\mathrm R})-\boldsymbol\varepsilon(\mathbf u)+\nu^{-1} p\,\mathrm I\Big)+ A\,(2\boldsymbol\varepsilon(\mathbf u)-\nu^{-1}\,p\,\mathrm I).
\end{eqnarray*}
Taking the jump of the normal components on both sides of the previous formula, it follows that
\[
\mathbf 0=\mu^{-1} \jump{\boldsymbol\sigma^{\mathrm L}(\mathbf u^{\mathrm L})}= 2 A
\jump{(\boldsymbol\varepsilon(\mathbf u^\Delta-\mathbf u^{\mathrm R}-\mathbf u)+ \nu^{-1} p\,\mathrm I)\mathbf n}+\nu^{-1} \jump{\mathbf t(\mathbf u,p)}.
\]
Since this results holds for all $0<A<1$, then $\jump{\mathbf t(\mathbf u,p)}=\mathbf 0$.

If we now integrate $\mathbf u$ on the curve/surface $\Xi=\partial B(\mathbf 0;R)$ and note that --with the notation of \eqref{eq:9.50} and \eqref{eq:9.51}--
\begin{eqnarray*}
\int_\Xi \mathbf u(\mathbf x)\mathrm d\Xi(\mathbf x) &=& \int_\Xi\langle \mathrm E_u(\mathbf x-\punto),\boldsymbol\lambda\rangle_\Gamma\mathrm d\Xi(\mathbf x) -\int_\Xi \Big( \int_\Gamma \mathrm T_u(\mathbf x-\mathbf y;\mathbf n(\mathbf y))\boldsymbol\varphi(\mathbf y)\mathrm d\Gamma(\mathbf y)\Big) \mathrm d\Xi(\mathbf x)\\
& = & \left\langle \int_\Xi \mathrm E_u(\mathbf x-\punto)\mathrm d\Xi(\mathbf x), \boldsymbol\lambda\right\rangle_{\!\Gamma}\!\!\!-\!\!
\int_\Gamma \!\Big( \int_\Xi \mathrm T_u(\mathbf x-\mathbf y;\mathbf n(\mathbf y))\mathrm d\Xi(\mathbf x)\Big) \boldsymbol\varphi(\mathbf y)\mathrm d\Gamma(\mathbf y)\\
&=& \langle \mathrm B_{\mathrm S},\boldsymbol\lambda\rangle_\Gamma-\int_\Gamma \mathrm B_{\mathrm D}(\mathbf y)\boldsymbol\varphi(\mathbf y)\mathrm d\Gamma(\mathbf y)=\boldsymbol\ell_{\mathrm S}(\boldsymbol\lambda)-\boldsymbol\ell_{\mathrm D}(\boldsymbol\varphi),
\end{eqnarray*}
then we get condition \eqref{eq:9.52} in the two dimensional case.
\end{proof}

\paragraph{In conclusion.} Proposition \ref{prop:11.3} shows that the potentials defined with integral formulas \eqref{eq:11.20} and \eqref{eq:11.21} are the same as the potentials defined through variational problems (Sections \ref{sec:5.1}, \ref{sec:6.1} and \ref{sec:2d}). Furthermore, formulas \eqref{eq:11.1},  \eqref{eq:11.6}, and \eqref{eq:11.7} give simple expressions of the Stokes potentials in terms of Lam\'e and Laplace potentials, with the fleeting presence of the rotlet \eqref{eq:11.22}. Our story ends here, but this is far from all that can be done with the theory of these operators. The formulas \eqref{eq:11.6} and \eqref{eq:11.7} (and the parametric analysis in terms of the non-physical quantity $0$) give a fast track transfer of many results from the better understood theory of Lam\'e and Laplace layer potentials to the Stokes world.

\bibliographystyle{abbrv}
\bibliography{referencesStokesBEM}

\end{document}